\documentclass[reqno, 12pt]{amsart}
\usepackage{amsmath,amsxtra,amssymb,latexsym, amscd,amsthm}
\usepackage[mathscr]{euscript}
\usepackage{mathrsfs}
\usepackage{color}
\usepackage[active]{srcltx}
\usepackage{mathtools,stmaryrd}
\usepackage{enumerate}
\usepackage{hyperref}
\usepackage{framed}
\usepackage[T1]{fontenc}
\usepackage{subfigure}
\usepackage{multirow}
\usepackage{booktabs}

\newcommand{\RR}{{\mathbb R}}

\newcommand{\N}{\mathbb{N}}

\newcommand{\bx}{x}

\newcommand{\by}{y}
\newcommand{\K}{\mathbf{K}}
\newcommand{\Y}{\mathbf{Y}}

\newcommand{\bs}{\mathbf{S}}
\newcommand{\bX}{x}
\newcommand{\bY}{y}
\newcommand{\ud}{\mathrm{d}}
\newcommand{\td}[1]{\tilde{#1}}

\newcommand{\dsos}{\mbox{\upshape\tiny dsos}}
\newcommand{\sdsos}{\mbox{\upshape\tiny sdsos}}
\newcommand{\psdp}{\mbox{\upshape\tiny primal}}
\newcommand{\dsdp}{\mbox{\upshape\tiny dual}}

\newcommand{\re}{\mathbb{R}}

\newcommand{\ckd}[1]{\mathcal{P}^{#1}(\mathbf{Y})}

\newcommand{\dist}{\text{\rm dist}}

\def\af{\alpha}

\newcommand{\mH}{\mathscr{H}}
\newcommand{\mL}{\mathscr{L}}
\newcommand{\wt}[1]{\widetilde{#1}}

\newcommand{\vol}{\text{\upshape vol}}

\newcommand{\bm}{\mathbf{M}}

\newtheorem{theorem}{Theorem}[section]

\newtheorem{corollary}{Corollary}[section]
\newtheorem{lemma}{Lemma}[section]
\newtheorem{example}{Example}[section]
\newtheorem{definition}{Definition}[section]
\newtheorem{remark}{Remark}[section]
\newtheorem{proposition}{Proposition}[section]

\def\EES{{\accent"5E e}\kern-.5em\raise.8ex\hbox{\char'23 }}
\def\ow{o\kern-.42em\raise.82ex\hbox{
   \vrule width .12em height .0ex depth .075ex \kern-0.16em \char'56}\kern-.07em}
\def\OW{o\kern-.460em\raise1.36ex\hbox{
\vrule width .13em height .0ex depth .075ex \kern-0.16em
\char'56}\kern-.07em}
\def\DD{D\kern-.7em\raise0.4ex\hbox{\char '55}\kern.33em}

\title{An SDP method for Fractional Semi-infinite Programming Problems
with SOS-convex polynomials}

\author{Feng Guo$^*$}
\address[Feng Guo]{School of Mathematical Sciences, Dalian  University of Technology, Dalian, 116024, China}
\email{fguo@dlut.edu.cn}
\thanks{$^*$Corresponding Author.}

\author{Meijun Zhang}
\address[Meijun Zhang]{School of Mathematical Sciences, Dalian  University of Technology, Dalian, 116024, China}
\email{mjzhang2021@163.com}

\date{ \today}

\begin{document}

\begin{abstract}
In this paper, we study a class of fractional
semi-infinite polynomial programming problems involving
sos-convex polynomial functions.
For such a problem, by a conic reformulation proposed in our
previous work and the quadratic modules associated with the index set,
a hierarchy of semidefinite programming (SDP) relaxations can be
constructed and convergent upper bounds of the optimum can be
obtained. In this paper, by introducing Lasserre's measure-based
representation of nonnegative polynomials on the index set to the
conic reformulation, we present a new SDP relaxation method for the
considered problem. This method enables us to compute convergent lower
bounds of the optimum and extract approximate minimizers. Moreover,
for a set defined by infinitely many sos-convex polynomial
inequalities,
we obtain a procedure to
construct a convergent sequence of outer approximations which have 
semidefinite representations (SDr). The convergence rate of the lower bounds
and outer SDr approximations are also discussed. 
\end{abstract}
\subjclass[2010]{65K05; 90C22; 90C29; 90C34}
\keywords{fractional optimization;
convex semi-infinite systems; semidefinite programming relaxations;
sum-of-squares convex; polynomial optimization}

\maketitle

\section{Introduction}
The fractional semi-infinite polynomial programming (FSIPP) problem considered in
this paper is in the following form:
\begin{equation}\label{FMP}
	\left\{
	\begin{aligned}
		r^{\star}:=\min_{x\in\RR^m}&\  \frac{f(x)}{g(x)}\\
		\text{s.t.}&\ \varphi_1(x)\le 0,\ldots,\varphi_s(x)\le 0, \\
		&\ p(\bx,\by)\le 0,\ \ \forall \by\in \Y\subset\RR^n,
	\end{aligned}\right. \tag{FSIPP}
\end{equation}
where $f, g$, $\varphi_1,\ldots,\varphi_s\in\RR[\bX]$ and
$p\in\RR[\bX,\bY]$. Here, $\RR[x]$ (resp. $\RR[x,y]$) denotes the ring
of real polynomials in $x=(x_1,\ldots,x_m)$ (resp., $x=(x_1,\ldots,x_m)$
and $y=(y_1,\ldots,y_n)$).
We denote by $\K$ and $\bs$ the feasible set and the set of optimal
solutions of \eqref{FMP}, respectively.
In this paper, we assume that $\bs\neq\emptyset$ and consider the
following assumptions on \eqref{FMP}:
\vskip 3pt
\noindent {\bf A1:}\ (i) $\Y\subseteq[-1, 1]^n$ and is closed; (ii)
		$\varphi_j$, $j=1,\ldots,s,$
		$p(\cdot, \by)$, $y\in \Y$ are all sos-convex;\\
\noindent {\bf A2:}\ (i) $f$, $-g$ are both sos-convex; (ii)
Either $f(x)\ge 0$ and $g(x)>0$ for all $x\in\K$, or $g(x)$
		is affine and $g(x)>0$ for all $x\in\K.$
%

A convex polynomial is called sos-convex if its Hessian matrix can
be written as the product of a polynomial matrix and its transpose
(see Definition \ref{sos-convex}). In particular, separable convex
polynomials and convex quadratic functions are sos-convex.
Therefore, our model \eqref{FMP} under {\bf A1-2}
contains subclasses of {\itshape linear semi-infinite programming} and {\itshape
convex quadratic semi-infinite programming} with polynomial
parametrizations.  Moreover, if $\Y$ is
defined by finitely many polynomial inequalities, then the problem of
minimizing a polynomial $h(y)\in\RR[y]$ over
$\Y$ can be reformulated as an FSIPP problem satisfying {\bf A1-2}.
As is well known, the polynomial optimization problem
is NP-hard even when $n>1$, $h(y)$ is a nonconvex
quadratic polynomial and $\Y$ is a polytope (c.f.
\cite{PardalosVavasis1991}).
Hence, in general the FSIPP problem considered in this paper cannot
be expected to be solved in polynomial time unless P=NP.
Particularly, minimizing a ratio of quadratic functions is of great
importance and some methods can be found in
\cite{SSY1995,WMX2021,YX2020}. However, these methods were given for
dealing with finitely constrained problems, while we aim to solve the
problem \eqref{FMP} with infinitely many constraints.

Over the last several decades, due to a great number of applications
in many fields, semi-infinite programming (SIP) has attracted a great deal
of interest and been very active research areas
\cite{Goberna2017,Goberna2018,Hettich1993,Still2007}.
Numerically, SIP problems can be solved by different approaches
including, for instance, discretization methods,
local reduction methods, exchange methods,
simplex-like methods etc; see \cite{Goberna2017,Hettich1993,Still2007}
and the references therein for details.
If the functions involved in SIP are polynomials, the 
representations of nonnegative polynomials over semi-algebraic sets
from real algebraic geometry allow us to derive semidefinite
programming (SDP) \cite{handbookSDP} relaxations for such problems
\cite{GS2020,Lasserre2011,SIPPSDP,XSQ}. 

In our previous work
\cite{GuoJiaoFSIPP}, instead of sos-convexity, we deal with the
FSIPP problems under convexity assumption. In \cite{GuoJiaoFSIPP}, we
first reformulate the FSIPP problem to a conic optimization problem. This
conic reformulation, together with inner approximations with
sums-of-square structures of the cone of nonnagative polynomials on
$\Y$ (e.g. the quadratic modules \cite{Putinar1993} associated with
$\Y$), enables us to derive a hierarchy of SDP relaxations of
\eqref{FMP}.  Applying such appoach to \eqref{FMP} under 
sos-convexity assumption, we can obtain convergent {\itshape upper}
bounds of $r^{\star}$. In this paper, we follow the methodology in
\cite{GuoJiaoFSIPP} and present a new SDP method for \eqref{FMP} under {\bf
A1-2}. Instead of the quadratic modules associated with $\Y$, we
introduce Lasserre's measure-based representation of nonnegative polynomials
on $\Y$ (c.f. \cite{LasserreOuter}) to the conic reformulation of
\eqref{FMP}. With the new SDP method, we can compute convergent
{\itshape lower} bounds of $r^{\star}$ and extract approximate
minimizers of \eqref{FMP} in the case when $\Y$ is a simple set, like 
a box, a ball, a sphere, or a polytope. 

We say that a convex set $C$ in $\RR^m$ has a {\itshape semidefinite
representation} (SDr) if there exist some integers
$l, k$ and real $k\times k$ symmetric matrices $\{A_i\}_{i=0}^m$ and
$\{B_j\}_{j=1}^l$ such that
\begin{equation}\label{eq::sdr}
	C=\left\{x\in\RR^m\ \Big|\ \exists w\in\RR^l,\ \text{s.t.}\ A_0+\sum_{i=1}^m
	A_ix_i+\sum_{j=1}^l B_jw_j\succeq 0\right\}.
\end{equation}
Semidefinite representations of convex
sets can help us to build SDP relaxations of many computationally intractable
optimization problems. Arising from it, 
one of the basic issues in convex algebraic geometry is to
characterize convex sets in $\RR^m$ which are SDr sets and give
systematic procedures to obtain their semidefinite representations
(or arbitrarily close SDr approximations)
\cite{LecturesConOpt,thetabody,HeltonNie,SRCSHN,convexsetLasserre,Lasserre15,MHL15}.
Observe that the feasible set of \eqref{FMP} is a subset of $\RR^m$
defined by infinitely many sos-convex polynomial inequalities. For a
set of this form, applying the approach in our previous work
\cite{GS2020}, a convergent sequence of {\itshape inner} SDr
approximations can be constructed. In this paper, from the new SDP
relaxations of \eqref{FMP}, we obtain a procedure to construct a
convergent sequence of {\itshape outer} SDr approximations of such a
set. 

Remark that the main ingredient in our method to obtain the lower
bounds of $r^{\star}$ and the outer SDr approximations of $\K$ is
Lasserre's measure-based representation of nonnegative polynomials on
the index set \cite{LasserreOuter}. For polynomial minimization
problems which can be regarded as a special case of \eqref{FMP}, the
convergence rate of Lasserre's measure-based upper
bounds is well studied in \cite{SL2020} in difference situations. 
By combining the results in \cite{SL2020} and the metric regularity of
semi-infinite convex inequality system (c.f. \cite{CKLP2007}), 
we derive some convergence analysis of the lower bounds of $r^{\star}$
and the outer SDr approximations of $\K$ obtained in this paper.

\vskip 5pt
This paper is organized as follows. In Section \ref{sec::pre}, some
notation and preliminaries are given. In Section \ref{sec::SDP}, we
present a hierarchy of SDP relaxations for the lower bounds of
$r^{\star}$ and a procedure to construct a convergent sequence of
outer SDr approximations of $\K$. 
The convergence rate of the lower bounds and outer SDr approximations
is discussed in Section \ref{sec::rate}.
Some numerical experiments are given in Section \ref{sec::num}.

\section{Preliminaries}\label{sec::pre}

In this section, we collect some notation and preliminary results
which will be used in this paper.
We denote by $x$ (resp., $y$)
the $m$-tuple (resp., $n$-tuple) of variables $(x_1,\ldots,x_m)$ (resp.,
$(y_1,\ldots,y_n)$).
The symbol $\N$ (resp., $\re$, $\RR_+$) denotes
the set of nonnegative integers (resp., real numbers, nonnegative real
numbers). For any $t\in \re$, $\lceil t\rceil$ denotes the smallest
integer that is not smaller than $t$.
For $u\in \re^m$,
$\Vert u\Vert$ denotes the standard Euclidean norm of $u$.
For $\alpha=(\alpha_1,\ldots,\alpha_n)\in\N^n$,
$|\alpha|=\alpha_1+\cdots+\alpha_n$.
For $k\in\N$, denote $\N^n_k=\{\alpha\in\N^n\mid |\alpha|\le
k\}$ and $\vert\N^n_k\vert$ its cardinality.
For variables $x \in \re^m$, $y \in \re^n$ and
$\beta\in\N^m, \af \in \N^n$, $x^\beta$, $y^\af$ denote
$x_1^{\beta_1}\cdots x_m^{\beta_m}$,
$y_1^{\af_1}\cdots y_n^{\af_n}$, respectively.
$\RR[x]$ (resp., $\mathbb{R}[y]$) denotes
the ring of polynomials in $x$ (resp., $y$) with real
coefficients.
For $h \in \RR[x]$ (resp. $\in\RR[y]$), we denote by $\deg_x (h)$
(resp. $\deg_y(h)$) its (total) degree.
For $k\in\N$, denote by $\RR[x]_k$ (resp., $\RR[y]_k$) the set of polynomials
in $\RR[x]$ (resp., $\RR[y]$) of degree up to $k$.
For $A=\RR[x],\ \RR[y],\ \RR[x]_k,\ \RR[y]_k$, denote by $A^*$ the dual
space of linear functionals from $A$ to $\RR$.
Denote by $\mathbf{B}^m$ the unit ball in $\RR^m$ and
$\mathbf{B}^m_r(u)$ (resp., $\mathbf{B}^m_r$) the
ball centered at $u$ (resp., the origin) in $\RR^m$ with the radius $r$. 

\vskip 10pt 

One of the difficulties in solving \eqref{FMP} is the
feasibility test of a point $u\in\RR^m$, which is caused by the
infinitely many constraints $p(u,y)\le 0$ for all $y\in\Y$. 
Thus, it is reasonable to
study the representations of nonnegative polynomials on $\Y$.
Denote
\[
	d_y:=\deg_y
	p(x,y)\quad\text{and}\quad\mathcal{P}_{d_y}(\Y):=\{h(y)\in\RR[y]_{d_y}\mid 
	h(y)\ge 0, \ \forall y\in\Y\}. 
\]
Now we recall the measure-based outer approximations of
$\mathcal{P}_{d_y}(\Y)$ proposed by Lasserre \cite{LasserreOuter}.


A polynomial $h\in\RR[y]$ is said to be a
sum-of-squares (sos) of polynomials if it can be written as $h=\sum_{i=1}^l
h_i^2$ for some $h_1,\ldots,h_l \in\RR[y]$.
The symbols $\Sigma^2[x]$ and $\Sigma^2[y]$ denote the sets of polynomials that are
sum-of-squares of polynomials in $\RR[x]$ and $\RR[y]$, respectively. For each $k\in\N$,
denote $\Sigma^2_k[x]:=\Sigma^2[x]\cap\RR[x]_{2k}$ and
$\Sigma^2_k[y]=\Sigma^2[y]\cap\RR[y]_{2k}$, respectively.
Note that for a given $h\in\RR[y]$, checking if 
$h\in\Sigma^2_k[y]$ is an SDP feasibility problem.
In fact, denote by $\mathbf{v}_k$ the column vector containing all
monomials in $\RR[y]$ of degree at most $k$.  Then, $h\in\Sigma^2_k[y]$ if and only if there
exists a positive semidefinite matrix
$H\in\RR^{|\mathbf{v}_k|\times|\mathbf{v}_k|}$ such that
$h=\langle H, \mathbf{v}_k{\mathbf{v}_k}^T\rangle$ (c.f. \cite{convexAGch6}). 

In the rest of this paper, 
\begin{center}
	let $\nu$ be a fixed and finite Borel measure with support exactly $\Y$.
\end{center}
For each $k\in\N$, define
\begin{equation}\label{eq::outerapp}
	\mathcal{P}^{k}_{d_y}(\Y):=\left\{\psi(y)\in\RR[y]_{d_y} \mid
	\int_{\Y} \psi(y)\sigma(y) \ud\nu(y)\ge0, \ \forall
	\sigma\in\Sigma^2_k[y]\right\}.
\end{equation}
Then for each $k\in\N$, it is clear that $\mathcal{P}^{k}_{d_y}(\Y)$ is a
closed subset of $\RR[y]_{d_y}$ and
$\mathcal{P}^{k}_{d_y}(\Y)\supset\mathcal{P}_{d_y}(\Y)$. 

To lighten the notation, throughout
the rest of the paper, we abbreviate the notation
$\mathcal{P}_{d_y}(\Y)$ and $\mathcal{P}^{k}_{d_y}(\Y)$ to
$\mathcal{P}(\Y)$ and $\mathcal{P}^{k}(\Y)$, respectively. 
\begin{theorem}\cite[Theorem 3.2]{LasserreOuter}\label{th::outer}
	Suppose that $\Y\subseteq[-1, 1]^n$, then 
	we have
	$\ckd{k_1}\supset\ckd{k_2}\supset\mathcal{P}(\Y)$ for
	$k_1<k_2$ and $\mathcal{P}(\Y)=\cap_{k=1}^\infty\ckd{k}$. 
\end{theorem}
\begin{remark}{\rm
		As proved in \cite[Theorem 3.2]{LasserreOuter}, it is
		required that $\Y\subseteq[-1, 1]^n$ for the
		convergence result in Theorem \ref{th::outer}. That is why it
		is assumed in {\bf A1} which, however, can be fulfilled after
		a possible rescaling if $\Y$ is compact.
}
\end{remark}

For a given $\psi\in\RR[y]_{d_y}$, it is not hard to see that
$\psi\in\ckd{k}$ if and only if the matrix $\int_{\Y}\psi
\mathbf{v}_k{\mathbf{v}_k}^T\ud\nu(y)$ is positive semidefinite. Observe that
each entry in the matrix $\int_{\Y}\psi
\mathbf{v}_k\mathbf{v}_k^T\ud\nu(y)$
is a linear combination of the coefficients of $\psi$. It implies that
each $\ckd{k}$ has an SDr with no lifting ($l=0$ in \eqref{eq::sdr})
and thus checking if $\psi\in\ckd{k}$ is
an SDP feasibility problem. We emphasize that, to get the SDr of
$\ckd{k}$, we need compute effectively
the integrals $\int_{\Y}y^{\beta}\ud\nu(y)$, $\beta\in\N^n$. 
There are several interesting cases of $\Y$ where these integrals can be
obtained either explicitly in closed form or numerically (see
\cite{LasserreOuter} and Section \ref{sec::num}).

\vskip 10pt
Now let us recall some background above sos-convex polynomials in
$\RR[x]$ introduced by Helton and Nie \cite{SRCSHN}.
\begin{definition}{\upshape \cite{SRCSHN}}\label{sos-convex}
	A polynomial $h\in\RR[x]$ is sos-convex if there are an
	integer $r$ and a  matrix polynomial $H\in\RR[x]^{r\times m}$ such
	that the Hessian $\nabla^2 h=H(x)^TH(x)$.
\end{definition}

Clearly, an sos-convex polynomial is convex. However, the
converse is not true. Ahmadi and Parrilo \cite{Ahmadi2012} proved that
the set of convex polynomials and the set of sos-convex polynomials in
$\RR[x]_k$ coincide if and
only if $m=1$ or $k=2$ or $(m,k)=(2,4)$. 
Thus, any convex quadratic function and any convex separable
polynomial is an sos-convex polynomial. 
The significance of sos-convexity is that
it can be checked numerically by solving an SDP problem (see
\cite{SRCSHN}), while checking the convexity of a polynomial is generally
NP-hard (c.f. \cite{Ahmadi2013}).
Interestingly, an extended Jensen's inequality holds for sos-convex
polynomials.
\begin{proposition}{\upshape\cite[Theorem 2.6]{LasserreConvexity2009}}\label{prop::Jensen}
	Let $h\in\RR[x]_{2d}$ be sos-convex$,$ and let
	$\mL\in(\RR[x]_{2d})^*$ satisfy $\mL(1)=1$ and $\mL(\sigma)\ge 0$ for
	every $\sigma\in\Sigma^2_d[x]$. Then$,$ $$\mL(h(x))\ge
	h(\mL(x_1),\ldots,\mL(x_m)).$$
\end{proposition}

The following result plays a significant role
in this paper.
\begin{lemma}{\upshape\cite[Lemma 8]{SRCSHN}}\label{lem::sosconvex}
	Let $h\in\RR[x]$ be sos-convex. If $h(u)=0$ and $\nabla h(u)=0$
	for some $u\in\RR^m,$ then $h$ is an sos polynonmial.
\end{lemma}

\section{SDP relaxations of FSIPP}\label{sec::SDP}
In this section, we first recall the conic reformulation of
\eqref{FMP} proposed in our previous work \cite{GuoJiaoFSIPP}. This
conic reformulation, together with inner approximations with sos
structures of $\mathcal{P}(\Y)$ (e.g., the quadratic modules
\cite{Putinar1993} associated with $\Y$), allows us to derive a
hierarchy of SDP relaxations of \eqref{FMP} and obtain convergent
upper bounds of $r^{\star}$. As a complement, we apply in this paper
the outer approximations $\ckd{k}$ of
$\mathcal{P}(\Y)$ to the conic reformulation and get a new SDP
relaxation method of \eqref{FMP} which can give us convergent lower
bounds of $r^{\star}$. Moreover, we gain a convergent sequence of
outer SDr approximations of $\K$. 

\subsection{Conic reformulation}

In this subsection, let us recall the conic reformulation of
\eqref{FMP} proposed in \cite{GuoJiaoFSIPP} which makes it possible to
derive SDP relaxations of \eqref{FMP}.

Consider the problem
\begin{equation}\label{FMP2}
	\min_{x\in\K}\ \  f(x)-r^{\star}g(x).
\end{equation}
Note that, under {\bf A1-2}, \eqref{FMP2} is clearly a  convex
semi-infinite programming problem and its optimal value is $0$.
Denote by $\mathcal{M}(\Y)$ the set of finite nonnegative measures supported
on $\Y$. Then, the Lagrangian dual of \eqref{FMP2} reads
\begin{equation}\label{eq::dcsip}
	\max_{\mu\in\mathcal{M}(\Y),\eta_j\ge 0} \inf_{x\in\RR^m}
	L_{f,g}(x,\mu,\eta),
\end{equation}
where
\begin{equation}\label{eq::lag}
		L_{f,g}(x,\mu,\eta):=f(x)-r^{\star}g(x)+\int_{\Y}
		p(x,y)\ud\mu(y)+\sum_{j=1}^s\eta_j\varphi_j(x).
\end{equation}

Consider the assumption that

\noindent {\bf A3:}\ The Slater condition holds for $\K$, i.e., there exists
$u\in\K$ such that $p(u,y)<0$ for all $y\in\Y$ and $\varphi_j(u)<0$
for all $j=1,\ldots,s$. 
\begin{proposition}{\rm (c.f. \cite{Still2007,ShapiroSIP})}\label{prop::slater}
Under {\bf A1-3}, 
then there exist $\mu^{\star}\in\mathcal{M}(\Y)$ and
$\eta^{\star}\in\RR_+^s$ such that $\inf_{x\in\RR^m}
L_{f,g}(x,\mu^{\star},\eta^{\star})=0$. Moreover, $\mu^{\star}$ can be
chosen as an atomic measure, i.e., $\mu^{\star}=\sum_{i=1}^l
\lambda_i\zeta_{v_i}$ where $l\le n$, each $\lambda_i>0$ and $\zeta_{v_i}$ is
the Dirac measure at $v_i\in\Y$. 
\end{proposition}
%
%
%
%

Let $d_x:=\deg_{\bX}(p(\bX,\bY))$ and 
\begin{equation}\label{eq::d}
	\mathbf{d}:=\lceil\max\{\deg(f), \deg(g), \deg(\psi_1), \ldots,
		\deg(\psi_s), \deg_{\bX}(p(\bX,\bY))\}/2\rceil.
\end{equation}

%
%
%
%
For $\mL\in(\RR[x])^*$ (resp., $\mH\in(\RR[y])^*$), denote by
$\mL(p(x,y))$ (resp., $\mH(p(x,y))$) the image of $\mL$ (resp., $\mH$)
on $p(x,y)$ regarded as an element in $\RR[x]$ (resp., $\RR[y]$) with
coefficients in $\RR[y]$ (resp., $\RR[x]$), i.e.,
$\mL(p(x,y))\in\RR[y]$ (resp., $\mH(p(x,y))\in\RR[x]$).

Consider the following conic optimization problem
\begin{equation}\label{eq::primal}
\qquad	\left\{
	\begin{aligned}
		\hat{r}:=\sup_{\rho,\mH,\eta}\
		&\rho&\\
		\text{s.t.}\
		&f(\bX)-\rho
		g(x)+\mH(p(\bX,\by))+\sum_{j=1}^s\eta_j\varphi_j(x)\in\Sigma_{\mathbf{d}}^2[x],
		&\\
		&\rho\in\RR,\ \mH\in(\mathcal{P}(\Y))^*,\ \eta\in\RR_+^s.&
	\end{aligned}\right.
\end{equation}

\begin{proposition}\label{prop::eq}
	Under {\bf A1-3}, we have $\hat{r}=r^{\star}$. 
\end{proposition}
\begin{proof}
Let (atomic) $\mu^{\star}\in\mathcal{M}(\Y)$ and $\eta^{\star}\in\RR_+^s$ be
the dual variables in Proposition \ref{prop::slater}. 
Define $\mH^{\star}\in(\RR[y])^*$ by letting
$\mH^{\star}(y^\beta)=\int_{\Y} y^\beta d\mu^{\star}(\by)$ for any
$\beta\in\N^n$.  Then, $\mH^{\star}\in(\mathcal{P}(\Y))^*$.
Since $\mu^{\star}$ is atomic, it is easy to see that
$L_{f,g}(x,\mu^{\star},\eta^{\star})$ is sos-convex under {\bf
A1-2}. Then, Lemma \ref{lem::sosconvex}
implies that
$L_{f,g}(x,\mu^{\star},\eta^{\star})\in\Sigma_{\mathbf{d}}^2[x]$. 
Therefore, $(r^{\star},\mH^{\star},\eta^{\star})$ is feasible to
\eqref{eq::primal} and $\hat{r}\ge r^{\star}$. On the other
hand, for any $u^{\star}\in\bs$ and any $(\rho,\mH,\eta)$ feasible to
\eqref{eq::primal}, it holds that
\[
	f(u^{\star})-\rho
	g(u^{\star})+\mH(p(u^{\star},\by))+\sum_{j=1}^s\eta_j\psi_j(u^{\star})\ge
	0.
\]
Then, the feasibility of $u^{\star}$ to \eqref{FMP} implies that
$r^{\star}=\frac{f(u^{\star})}{g(u^{\star})}\ge\rho$ and thus
$r^{\star}\ge \hat{r}$. 
\end{proof}

\begin{remark}\label{rk::sos}{\rm
 In view of the proof of Proposition \ref{prop::eq}, if we replace
 $\Sigma_{\mathbf{d}}^2[x]$ in \eqref{eq::primal} by any convex cone
 $\mathcal{C}[x]\subset\RR[x]$ satisfying the condition 
 \[
	 \Sigma_{\mathbf{d}}^2[x]\subseteq\mathcal{C}[x]\quad\text{and
	 there exists } u^{\star}\in\bs\ \text{such that } h(u^{\star})\ge
	 0\ \text{for all } h\in\mathcal{C}[x],
 \]
 we still have $\hat{r}=r^{\star}$.	\qed}
\end{remark}

 If we substitute $\mathcal{P}(\Y)$ in
		\eqref{eq::primal} by its approximations with sos
		structures, then \eqref{eq::primal} can be reduced to SDP
		problems and becomes tractable. In particular, if we replace
		$\mathcal{P}(\Y)$ by the quadratic modules \cite{Putinar1993}
		generated by the defining polynomials of $\Y$, which are
		{\itshape inner} approximations of $\mathcal{P}(\Y)$,
		we can obtain upper bounds of
		$r^{\star}$ from the resulting SDP relaxations.
		See \cite{GuoJiaoFSIPP} for more details.  
		Our goal in this paper is to compute convergent lower
		bounds of $r^{\star}$ by SDP relaxations derived from
		\eqref{eq::primal}. It will be done by
		substituting $\mathcal{P}(\Y)$ with the outer approximations
		$\ckd{k}$ in \eqref{eq::outerapp}.

\subsection{SDP relaxations for lower bounds of
	$r^{\star}$}\label{subsec::SDP}

In the rest of this paper, let us fix a suffciently large
$\mathfrak{R}>0$ and a sufficiently small $g^{\star}>0$
such that 
\begin{equation}\label{eq::rg}
	\Vert u^{\star}\Vert \le \mathfrak{R} \quad \text{and}\quad
	g(u^{\star})\ge
	g^{\star} \quad \text{for some}\quad u^{\star}\in\bs.
\end{equation}
See \cite[Remark 4.1]{GuoJiaoFSIPP} for the choice of $\mathfrak{R}$ and
$g^{\star}$ in some circumstances.
Let 
\[
	Q=\{q_1(x):=\mathfrak{R}^2-\Vert x\Vert^2,\  \
	q_2(x):=g(x)-g^{\star}\},
\]
and
\[
	\bm_{\mathbf{d}}(Q):=\left\{\sum_{j=0}^2\sigma_jq_j\ \Big|\ q_0=1,\
	\sigma_j \in \Sigma^2[x], \, \deg(\sigma_jq_j)\le 2\mathbf{d},\
j=0, 1, 2\right\},
\]
i.e., $\bm_{\mathbf{d}}(Q)$ be the $\mathbf{d}$-th quadratic module
generated by $Q$ \cite{Putinar1993}. By Remark \ref{rk::sos}, we
still have $\hat{r}=r^{\star}$ if we replace
$\Sigma^2_{\mathbf{d}}[x]$ by $\bm_{\mathbf{d}}(Q)$ in
\eqref{eq::primal}. 

\vskip 10pt
Consider the following problem, where we replace $\Sigma^2_{\mathbf{d}}[x]$ and
$\mathcal{P}(\Y)$ in \eqref{eq::primal} by
$\bm_{\mathbf{d}}(Q)$ and $\ckd{k}$, respectively, 
\begin{equation}\label{eq::f*rlower}
\qquad	\left\{
	\begin{aligned}
		r_k^{\psdp}:=\sup_{\rho,\mH,\eta}\
		&\rho&\\
		\text{s.t.}\
		&f(\bX)-\rho
		g(x)+\mH(p(\bX,\by))+\sum_{j=1}^s\eta_j\varphi_j(x)\in\bm_{\mathbf{d}}(Q),
		&\\
		&\rho\in\RR,\ \mH\in\left(\ckd{k}\right)^*,\ \eta\in\RR_+^s.&
	\end{aligned}\right.\tag{$\mathrm{P}_k$}
\end{equation}
Its Lagrangian dual reads
\begin{equation}\label{eq::f*rduallower}
\qquad	\left\{
	\begin{aligned}
		r_k^{\dsdp}:=\inf_{\mathscr{L}\in(\RR[x]_{2\mathbf{d}})^*}\
		&\mathscr{L}(f)&\\
		\text{s.t.}\
		& \mathscr{L}\in(\bm_{\mathbf{d}}(Q))^*,\ \mathscr{L}(g)=1,&\\
		& -\mL(p(x,y))\in\ckd{k},\ \mL(\varphi_j)\le 0,\ j=1,\ldots,s.
	\end{aligned}\right. \tag{$\mathrm{D}_k$}
\end{equation}

For each $k\in\N$, recall that checking if $-\mL(p(x,y))\in\ckd{k}$
for a given $\mathscr{L}\in(\RR[x]_{2\mathbf{d}})^*$ is an SDP
feasibility problem. Therefore, 
computing $r^{\psdp}_k$ and $r^{\dsdp}_k$ is reduced to solving a pair
of an SDP problem and its dual. We omit the detail for simplicity. In
the following, we will show that $\{r_k^{\psdp}\}_{k\in\N}$ and
$\{r_k^{\dsdp}\}_{k\in\N}$ are convergent lower bounds of $r^{\star}$, and we
can extract approximate minimizers of \eqref{FMP} from the SDP
relaxations \eqref{eq::f*rduallower}. 
To this end, we first point out that the feasible set of the
\eqref{eq::f*rduallower} is {\itshape uniformly} bounded. 
\begin{proposition}\label{prop::bound}
	For any $\mL\in(\RR[x]_{2\mathbf{d}})^*$ satisfying that
	$\mL(\sigma_0)\ge 0$ for any $\sigma_0\in\Sigma^2_{\mathbf{d}}[x]$
	and $\mL((\mathfrak{R}^2-\Vert x\Vert^2)\sigma)\ge 0$ for any  
	$\sigma\in\Sigma^2_{\mathbf{d}-1}[x]$, we have
	\[
\Vert(\mL(x^\alpha))_{\alpha\in\N^m_{2\mathbf{d}}}\Vert\le
\mL(1) \sqrt{\binom{m+\mathbf{d}}{m}}\sum_{i=0}^\mathbf{d}
	\mathfrak{R}^{2i}.
	\]
	Consequently, for any $k\in\N$ and any
	$\mL_k\in(\RR[x]_{2\mathbf{d}})^*$ feasible to
	\eqref{eq::f*rduallower},
	\[
\Vert(\mL_k(x^\alpha))_{\alpha\in\N^m_{2\mathbf{d}}}\Vert
\le \frac{1}{g^{\star}} \sqrt{\binom{m+\mathbf{d}}{m}}\sum_{i=0}^\mathbf{d}\mathfrak{R}^{2i}.
	\]
\end{proposition}
\begin{proof}
	For any $\mL\in(\RR[x]_{2\mathbf{d}})^*$ satisfying that
	$\mL(\sigma_0)\ge 0$ for any $\sigma_0\in\Sigma^2_{\mathbf{d}}[x]$
	and $\mL((\mathfrak{R}^2-\Vert x\Vert^2)\sigma)\ge 0$ for any  
	$\sigma\in\Sigma^2_{\mathbf{d}-1}[x]$,  by \cite[Lemma
	3]{CD2016} and its proof, it holds that
\[
	\sqrt{\sum_{\alpha\in\N^m_{2\mathbf{d}}}\left(\mL(x^{\alpha})\right)^2}
	\le\mL(1) \sqrt{\binom{m+\mathbf{d}}{m}}\sum_{i=0}^\mathbf{d}
	\mathfrak{R}^{2i}.
\]
For any $\mL_k\in(\RR[x]_{2\mathbf{d}})^*$ feasible to
\eqref{eq::f*rduallower}, we have $\mL_k(g-g^{\star})\ge 0$ and
$\mL_k(g)=1$. Hence, $\mL_k(1)\le\mL_k(g)/g^{\star}=1/g^{\star}$.
The conlusion follows.
\end{proof}

The following theorem states that we can compute convergent lower bounds
of $r^{\star}$ and extract approximate minimizers of \eqref{FMP} from
the SDP relaxations \eqref{eq::f*rlower} and \eqref{eq::f*rduallower}.
For any $\mL\in(\RR[x]_{2\mathbf{d}})^*$, denote 
\[
	\mL(x):=(\mL(x_1),\ldots,\mL(x_m))\in\RR^m. 
\]
\begin{theorem}\label{th::lower}
	Under {\bf A1-3}, it holds that
	\begin{enumerate}[\upshape (i)]
		\item $r^{\psdp}_k=r^{\dsdp}_k\le r^{\star}$ and
			$r^{\dsdp}_k$ is attainable for each $k\in\N$;
		\item $\lim_{k\rightarrow\infty}
			r^{\psdp}_k=\lim_{k\rightarrow\infty}
			r^{\dsdp}_k=r^{\star}$;
		\item For any convergent subsequence
			$\{\mL^{\star}_{k_i}(x)/\mL^{\star}_{k_i}(1)\}_i$
			$($always exists$)$ of
			$\{\mL^{\star}_k(x)/\mL^{\star}_k(1)\}_k$  where $\mL_k^{\star}$ is a minimizer of
			\eqref{eq::f*rduallower}, we have
			$\lim_{i\rightarrow\infty}\mL^{\star}_{k_i}(x)/\mL^{\star}_{k_i}(1)\in\bs$.
			Consequently$,$ if $\bs$ is singleton$,$ then 
			$\lim_{k\rightarrow\infty}\mL^{\star}_k(x)/\mL^{\star}_k(1)$
			is the unique minimizer of \eqref{FMP}. 
	\end{enumerate}
\end{theorem}
\begin{proof}
	(i) Fix a $u^{\star}\in\mathbf{S}$ satisfying \eqref{eq::rg} and
	define a linear functional $\mL^{\star}\in(\RR[x]_{2\mathbf{d}})^*$ by letting
	$\mL^{\star}(x^\alpha)=\frac{(u^{\star})^\alpha}{g(u^{\star})}$ for each
	$\alpha\in\N^m_{2\mathbf{d}}$. By the definition of $\bm_{\mathbf{d}}(Q)$
	and $\ckd{k}$, as
	well as Theorem \ref{th::outer}, it is easy to see that
	$\mL^{\star}$ is feasible to \eqref{eq::f*rduallower} for each
	$k\in\N$. Then,
	\[
		r^{\dsdp}_k\le
		\mL^{\star}(f)=\frac{f(u^{\star})}{g(u^{\star})}=r^{\star}.
\]
Then for any $k\in\N$, by Proposition \ref{prop::bound}, the feasible set of 
\eqref{eq::f*rduallower} is nonempty, uniformly bounded and
closed. Hence, the solution set of 
\eqref{eq::f*rduallower} is nonempty and bounded, which implies that
\eqref{eq::f*rlower} is strictly feasible (c.f. \cite[Section
4.1.2]{SS2000}). Consequently, the strong duality
$r_k^{\psdp}=r_k^{\dsdp}$ holds by \cite[Theorem 4.1.3]{SS2000}.

Now we show (ii) and (iii) together. 
Let
$\{\mL^{\star}_k\}_{k\in\N}\subset(\RR[x]_{2\mathbf{d}})^*$ be a sequence such
that $\mL^{\star}_k$ is a minimizer of \eqref{eq::f*rduallower} for each
$k\in\N$. 
As $\{\mL^{\star}_k(x^{\alpha})_{\alpha\in\N^m_{2\mathbf{d}}}\}_k$ is
	uniformly bounded by Proposition
	\ref{prop::bound}, there is a subsequence
	$\{\mL^{\star}_{k_i}\}_{k_i}$ and a
	$\mL^{\star}\in(\RR[x]_{2\mathbf{d}})^*$ such that
	$\lim_{i\rightarrow\infty}
	\mL^{\star}_{k_i}(x^\alpha)=\mL^{\star}(x^\alpha)$ for all
	$\alpha\in\N^m_{2\mathbf{d}}$. 
Because the sequence $\{r_k^{\dsdp}\}_k$ is monotone nondecreasing
and bounded by $r^{\star}$ as $k\rightarrow\infty$ , the limit of
$\{r_k^{\dsdp}\}_k$ exists and 
$\mL^{\star}(f)=\lim_{k\rightarrow\infty}r_k^{\dsdp}$. 
Moreover, from the pointwise convergence, we get the
following: (a) $\mL^{\star}\in(\bm_{\mathbf{d}}(Q))^*$;
(b) $\mL^{\star}(g)=1$; (c)
$-\mL^{\star}(p(x,y))\in\cap_{k=1}^\infty\ckd{k}$; (d)
$\mL^{\star}(\varphi_j)\le 0$ for $j=1,\ldots,s$. 
In particular, (c) holds because
	$\ckd{k}$ is closed in $\RR[y]_{d_y}$ and
	$\ckd{k_2}\subseteq\ckd{k_1}$ for $k_1<k_2$.
We have $\mL^{\star}(1)>0$. In fact, $\mL^{\star}(1)\ge 0$ since
$\mL^{\star}\in(\Sigma^2_{\mathbf{d}}[x])^*$ by (a). 
If $\mL^{\star}(1)=0$, then by Proposition \ref{prop::bound},
we have $\mL^{\star}(x^{\alpha})=0$
for all $\alpha\in\N^m_{2\mathbf{d}}$, which contradicts (b). From (c)
and Theorem \ref{th::outer}, we get
$-\mL^{\star}(p(x,y))\in\mathcal{P}(\Y)$. Then, for any $y\in\Y$, by Proposition
\ref{prop::Jensen}, 
\begin{equation*}\label{eq::J1}
	p\left(\frac{\mathscr{L}^{\star}(\bX)}{\mathscr{L}^{\star}(1)},\by\right)
	\le\frac{1}{\mL^{\star}(1)}\mathscr{L}^{\star}(p(\bX,\by))\le 0,
\end{equation*}
For the same reason, (d) implies that 
\begin{equation*}\label{eq::J2}
	\psi_{j}\left(\frac{\mathscr{L}^{\star}(\bX)}{\mathscr{L}^{\star}(1)}\right)\le
	0,\ \ j=1,\ldots,s,
\end{equation*}
which shows that
$\mathscr{L}^{\star}(\bX)/\mathscr{L}^{\star}(1)\in\K$. 
Since $f(\bX)$ and $-g(x)$ are also sos-convex, under {\bf A2},
we have
\[
	r^{\star}\le\frac{f\left(\frac{\mL^{\star}(x)}{\mL^{\star}(1)}\right)}{g\left(\frac{\mL^{\star}(x)}{\mL^{\star}(1)}\right)}
	\le
	\frac{\frac{1}{\mL^{\star}(1)}\mL^{\star}(f)}{\frac{1}{\mL^{\star}(1)}\mL^{\star}(g)}=\mL^{\star}(f)=\lim_{k\rightarrow\infty}r_k^{\dsdp}\le r^{\star}.
\]
It implies that $\frac{\mL^{\star}(x)}{\mL^{\star}(1)}\in\bs$ and
$\lim_{k\rightarrow\infty} r^{\psdp}_k=\lim_{k\rightarrow\infty}
r^{\dsdp}_k=r^{\star}$.

Assume that $\bs$ is singleton and let $\bs=\{u^{\star}\}$. 
The above arguments show that $\lim_{i\rightarrow\infty}
\mL^{\star}_{k_i}(x)/\mL^{\star}_{k_i}(1)=u^{\star}$ for any
convergent subsequence of $\{\mL^{\star}_k(x)/\mL^{\star}_k(1)\}_k$
which is bounded. Hence, the whole sequence
$\{\mL^{\star}_k(x)/\mL^{\star}_k(1)\}_k$ converges to $u^{\star}$ as
$k$ tends to $\infty$. 
\end{proof}
\begin{remark}\label{rk::lower}
	{\rm
	From its proof, we can see that Theorem \ref{th::lower} (i) still
	holds provided only {\bf A1}-(i) and the existence of
	$u^{\star}\in\bs$ satisfying \eqref{eq::rg}, while the convexity
	of $f$, $-g$, $\varphi_j$'s and $p(\cdot, y)$, $y\in\Y$, is not
	necessary.
}
\end{remark}


\subsection{Outer SDr approximations of $\K$}
Observe that the feasible set $\K$ of \eqref{FMP} is defined by
infinitely many sos-convex polynomial inequalities. For a
set of this form, applying the approach in our previous work
\cite{GS2020}, a convergent sequence of {\itshape inner} SDr
approximations can be constructed. This appoach relies on the sos
representation of the Lagrangian function
$L_{f,g}(x,\mu^{\star},\eta^{\star})$ and the quadratic modules
associated with $\Y$. Next, we show that a convergent sequence of
{\itshape outer} SDr approximations of $\K$ can be constructed
from the SDP relaxations \eqref{eq::f*rduallower}. 
For each $k\in\N$, define 
\begin{equation}\label{eq::outerappro}
	\Lambda_k:=\left\{\mL(x)\in\RR^m:
	\left\{
		\begin{aligned}
			&\mL(\sigma_0)\ge 0,\ \forall\
			\sigma_0\in\Sigma^2_{\mathbf{d}}[x],\\
			& \mL((\mathfrak{R}^2-\Vert x\Vert^2)\sigma)\ge 0,\
			\forall\ \sigma\in\Sigma^2_{\mathbf{d}-1}[x],\\
			&-\mL(p(x,y))\in\ckd{k},\  \mL(1)=1,\\
			&\mL(\varphi_j)\le 0,\ j=1,\ldots,s.
		\end{aligned}
	\right.\right\}.
\end{equation}

It is easy to see that $\Lambda_k$ is indeed an SDr set for each
$k\in\N$.

\begin{theorem}\label{th::outersdr}
	Under {\bf A1}, we have
	$\K\cap\mathbf{B}^m_{\mathfrak{R}}\subseteq\Lambda_{k_2}\subseteq\Lambda_{k_1}\subseteq\mathbf{B}^m_{\mathfrak{R}}$
	for $k_1<k_2$ and $\K\cap\mathbf{B}^m_{\mathfrak{R}}=\cap_{k=1}^\infty
	\Lambda_{k}$. Consequently, if $\K$ is compact and
	$\mathfrak{R}$ is
	large enough such that $\K\subset\mathbf{B}^m_{\mathfrak{R}}$, then
	$\K=\cap_{k=1}^\infty \Lambda_{k}$.
\end{theorem}
\begin{proof}
	It is clear that $\Lambda_{k_2}\subseteq\Lambda_{k_1}$
	for $k_1<k_2$. For any $u\in\K\cap\mathbf{B}^m_{\mathfrak{R}}$, 
	let $\mL'\in(\RR[x]_{2\mathbf{d}})^*$ be such
	that $\mL'(x^\alpha)=u^\alpha$ for each
	$\alpha\in\N^m_{2\mathbf{d}}$. Then by
	Theorem \ref{th::outer}, $\mL'$ satisfies the
	conditions in \eqref{eq::outerappro} and hence $u\in\Lambda_k$
	for each $k\in\N$. Assume that $\mL\in(\RR[x]_{2\mathbf{d}})^*$
	satisfies the conditions in \eqref{eq::outerappro} . As the function
	$\Vert x\Vert^2$ is sos-convex, by Proposition \ref{prop::Jensen}, 
	\[
		\Vert\mL(x)\Vert^2\le \mL(\Vert
		x\Vert^2)\le\mL(\mathfrak{R}^2)=\mathfrak{R}^2\cdot\mL(1)=\mathfrak{R}^2.
	\]
	Hence, $\Lambda_k\subseteq\mathbf{B}^m_{\mathfrak{R}}$ for all
	$k\in\N$. 

	It remains to prove that $\cap_{k=1}^\infty
	\Lambda_{k}\subseteq\K$.
	Fix a point $u\in\cap_{k=1}^\infty \Lambda_{k}$. Then for
	each $k\in\N$, there exsits a $\mL_k\in(\RR[x]_{2\mathbf{d}})^*$
	satisfying the conditions in \eqref{eq::outerappro} and
	$\mL_k(x)=u$. By Proposition \ref{prop::bound}, the vector 
	$(\mL_k(x^\alpha))_{\alpha\in\N^m_{2\mathbf{d}}}$ is uniformly
	bounded for all $k\in\N$. Then there exists a convergent subsequence
	$\{\mL_{k_i}\}_i$ and a $\wt{\mL}\in(\RR[x]_{2\mathbf{d}})^*$ such
	that
	$\lim_{i\rightarrow\infty}\mL_{k_i}(x^\alpha)=\wt{\mL}(x^\alpha)$
	for each $\alpha\in\N^m_{2\mathbf{d}}$. By the pointwise
	convergence, we obtain that (a) $\wt{\mL}(\sigma)\ge 0$ for all
	$\sigma\in\Sigma^2_{\mathbf{d}}[x]$; (b) $\wt{\mL}(1)=1$;
	(c) $-\wt{\mL}(p(x,y))\in\ckd{k}$ for each
	$k\in\N$; (d) $\wt{\mL}(\varphi_j)\le
	0,\ j=1,\ldots,s.$  
	By (c) and Theorem \ref{th::outer},
	$-\wt{\mL}(p(x,y))\in\mathcal{P}(\Y)$. Then for any $y\in\Y$, by
	the sos-convexity of $p(x,y)$ in $x$, (a), (b) and Proposition
	\ref{prop::Jensen} again,
	\[
		p(u,y)=p(\wt{\mL}(x), y)\le \wt{\mL}(p(x,y))\le 0. 
	\]
	For the same reason, we have $\varphi_j(u)\le 0$ for $j=1,\ldots,s$.
	We can conclude that $u\in\K$. 
\end{proof}


\begin{remark}{\rm
	In \cite{Lasserre15,MHL15}, some tractable methods using
	SDP are proposed to approximate
	semialgebraic sets defined with quantifiers. Clearly, the set
	$\K$ studied in this paper is in such a case with a universal
	quantifier. The method in \cite{Lasserre15,MHL15} works for $\K$
	in a general form without requiring $-p(x,y)$ to be
	convex in $x$ and approximates $\K$ by a sequence of sublevel
	sets of a single polynomial. Different from that, we construct
	convergent SDr approximations of $\K$ by fully exploiting the
sos-convexity of the defining polynomials.}\qed
\end{remark}

\subsection{Some discussions}
Typically, lower bounds of semi-infinite programming problems can be
computed by the discretization method by grids (see
\cite{Hettich1993}). Compared with other numerical methods for general
semi-infinite programming problems, this method can avoid globally
solving the lower level problem $\max_{y\in\Y} p(u,y)$ to test the
feasibility of a point $u\in\RR^m$, which could be very hard and is
one of the main computational problems in semi-infinite programming. 

Precisely, for \eqref{FMP}, we
can replace $\Y$ by $\Y\cap T$ where $T\subset [-1,1]^n$ is a fixed grid,
and solve the resulting finitely constrained problem
\begin{equation}\label{DFMP}
	\left\{
	\begin{aligned}
		\min_{x\in\RR^m}&\  \frac{f(x)}{g(x)}\\
		\text{s.t.}&\ \varphi_1(x)\le 0,\ldots,\varphi_s(x)\le 0, \\
		&\ p(\bx,\by)\le 0,\ \ \forall \by\in \Y\cap T.
	\end{aligned}\right. 
\end{equation}
We suppose that the Hausdorff distance between $\Y$ and $\Y\cap T$
tends to $0$ as the grid size of $T$ vanishes. Denote by $\K_T$ the
feasible set of \eqref{DFMP}. Then,
under {\bf A2}-(ii), we can assume that the grid size of $T$ is small
enough and hence $g(x)>0$ on $\K_T$. In fact, for a fixed $u\in\RR^m$ with $g(u)\le 0$, if
$\varphi_i(u)\le 0$ for all $i=1,\ldots,s$, then there must be a
point $\bar{y}\in\Y$ such that $p(u,\bar{y})>0$ because of {\bf
A2}-(ii). As the grid size of $T$ is small enough, there exists a point
$\hat{y}\in \Y\cap T$ close to $\bar{y}$ such that $p(u,\hat{y})>0$
which implies that $u\not\in\K_T$. 

We can consider the following
three ways to solve \eqref{DFMP} in the case when $g(x)$ is
affine. In this case, as $g(x)>0$ on $\K_T$,
it is not hard to check that
$\frac{f(x)}{g(x)}$ is strictly quasiconvex on $\K_T$ under {\bf
A1-2}. That is, for any $u, v\in\K_T$,
\[
	\frac{f(u)}{g(u)}<\frac{f(v)}{g(v)} \quad\text{implies}\quad 
\frac{f(\lambda_1 u+\lambda_2 v)}{g(\lambda_1 u+\lambda_2
v)}<\frac{f(v)}{g(v)}\quad\text{for any}\quad \lambda_1, \lambda_2>0
\quad\text{with}\quad \lambda_1+\lambda_2=1.
\]
Hence, the first way to solve \eqref{DFMP}, as a quasiconvex
optimization problem, is by using bisection method with each step a
convex feasibility problem. Second, since any
local minimizer of \eqref{DFMP} is also a global one (c.f.
\cite[Theorem 2]{Ponstein1967}) due to the strict quasiconvexity of
$\frac{f(x)}{g(x)}$ on $\K_T$, any local or global
methods (e.g. interior-point methods, SQP methods, etc.) for solving
general constrained nonlinear programming can be applied to \eqref{DFMP}. 
Third, we can also reformulate \eqref{DFMP} to an SDP problem
under the assumption that the Slater condition
holds for \eqref{DFMP}.  In fact, as $g(x)>0$ on $\K_T$, \eqref{DFMP} is
equivalent to 
\[
	\max_{r\in\RR}\  r\quad\text{s.t.}\quad f(x)-r g(x)\ge
	0\quad \text{for
		all}\quad x\in \K_T.
\]
Since $g(x)$ is affine, $f(x)-r g(x)$ is sos-convex for any $r\in\RR$.
Then, the convex
positivstellensatz \cite[Theorem 3.3]{LasserreConvexity2009} implies
that \eqref{DFMP} can be equivalently reformulated to 
\begin{equation}\label{eq::SDP4DFMP}
\left\{
	\begin{aligned}
		\max_{r\in\RR}&\  r\\
		\text{s.t.}&\ f(x)-r
		g(x)=\sigma+\sum_{i=1}^s\lambda_i\varphi_i(x)+\sum_{y\in\Y\cap T}
		\eta_y p(x,y),\\
		&\ \sigma\in\Sigma^2_{\mathbf{d}}[x],\ \lambda_i\ge 0,\
		i=1,\ldots,s, \ \eta_y\ge 0,\ y\in\Y\cap T,
	\end{aligned}\right. 
\end{equation}
which in fact is an SDP problem. Note that the number of nonnegative
variables $\eta_y$ is equal to the cardinality of $\Y\cap T$.

Convergent lower bounds of $r^{\star}$ can be obtained by solving
\eqref{DFMP} provided that mesh size
of the expansive sequence of grids tends to zero. However, in general,
it is challenging to generate effcient grids for such a task. For
a large $n$, if we use the regular grids 
\begin{equation}\label{eq::grid}
	T_N:=\left\{-1+\frac{2}{N}i\right\}_{i=0,\ldots,N}\times\cdots\times\left\{-1+\frac{2}{N}i\right\}_{i=0,\ldots,N}\subset
	[-1,1]^n, \quad N\in\N, 
\end{equation}
the rapidly increasing grid points in $\Y$ as $N$ increases cause the
resulting problems more and more intractable. See Example \ref{ex::compare} for a
comparison of our SDP method with the above discretization scheme.

\vskip 10pt
To end this section, we consider the possibility of applying the
{\itshape diagonally
dominant sum of squares} (dsos) and {\itshape scaled diagonally
dominant sum of squares} (sdsos) structures
\cite{AM2014,AM2019,AMT2014}to \eqref{eq::f*rlower}
for handling \eqref{FMP} problems with large
numbers $m$ and $\mathbf{d}$. For such problems, the sos structures in
the quadratic module $\bm_{\mathbf{d}}(Q)$ give rise to semidefinite
constraints of very large size in \eqref{eq::f*rlower} and
\eqref{eq::f*rduallower}, even when the order $k$ is small. 
In view of the capability of the
state-of-the-art SDP solvers, it can cause the resulting SDP problems
very hard to solve or even intractable. In this case, we may impose
the dsos and sdsos structures into \eqref{eq::f*rlower} to trade off
computation time with lower bound quality.

A symmetric matrix $A = (a_{ij})$ is diagonally dominant (dd) if
$a_{ii}\geq \sum_{j\neq i}|a_{ij}|$ for all $i$. A symmetric matrix
$A$ is scaled diagonally dominant (sdd) if there exists a diagonal matrix
$D$, with positive diagonal entries, such that $DAD$ is diagonally
dominant. A polynomial $h\in\RR[x]$ of degree $2d$ is dsos (resp. sdsos) if and
only if it admits a representation as $h(x) = z^T(x)Hz(x)$, where
$z(x)$ is the standard monomial vector of degree $\leq d$ in
$\RR[x]$ and $H$ is a dd (resp. sdd) matrix. 
We denote the set of polynomials in $\RR[x]_{2d}$ that are dsos (resp.
sdsos) by $DSOS_{m,2d}$ (resp. $SDSOS_{m,2d}$).  
It is clear that $DSOS_{m,2d}\subseteq SDSOS_{m,2d}\subseteq
\Sigma_d^2[x]$. In general, all these containment relationships are
strict. Notice that optimization over $DSOS_{m,2d}$ (resp.
$SDSOS_{m,2d}$) can be done with a linear program (resp. second-order cone
program) of size polynomial in $m$ (see \cite[Theorem 3.9]{AM2019}). 

Now we replace the sos structure in the quadratic module
$\bm_{\mathbf{d}}(Q)$ by dsos and sdsos structures, respectively, and
define the following cones
\[
	\bm_{\mathbf{d}}^{\dsos}(Q):=\left\{\sum_{j=0}^2\sigma_jq_j\ \Big|\ q_0=1,\
	\sigma_j \text{ is dsos}, \, \deg(\sigma_jq_j)\le 2\mathbf{d},\
j=0, 1, 2\right\},
\]
and
\[
	\bm_{\mathbf{d}}^{\sdsos}(Q):=\left\{\sum_{j=0}^2\sigma_jq_j\ \Big|\ q_0=1,\
	\sigma_j \text{ is sdsos}, \, \deg(\sigma_jq_j)\le 2\mathbf{d},\
j=0, 1, 2\right\}.
\]
Clearly, it holds that
\[
	\bm_{\mathbf{d}}^{\dsos}(Q)\subseteq\bm_{\mathbf{d}}^{\sdsos}(Q)\subseteq\bm_{\mathbf{d}}(Q).
\]
Replacing $\bm_{\mathbf{d}}(Q)$ in
\eqref{eq::f*rlower} by $\bm_{\mathbf{d}}^{\dsos}(Q)$ and
$\bm_{\mathbf{d}}^{\sdsos}(Q)$, respectively, we obtain
\begin{equation}\label{eq::dsos}
\qquad	\left\{
	\begin{aligned}
		r_k^{\dsos}:=\sup_{\rho,\mH,\eta}\
		&\rho&\\
		\text{s.t.}\
		&f(\bX)-\rho
		g(x)+\mH(p(\bX,\by))+\sum_{j=1}^s\eta_j\varphi_j(x)\in
		\bm_{\mathbf{d}}^{\dsos}(Q),
		&\\
		&\rho\in\RR,\ \mH\in\left(\ckd{k}\right)^*,\ \eta\in\RR_+^s,&
	\end{aligned}\right.\tag{$\mathrm{P}^{\dsos}_k$}
\end{equation}
and 
\begin{equation}\label{eq::sdsos}
\qquad	\left\{
	\begin{aligned}
		r_k^{\sdsos}:=\sup_{\rho,\mH,\eta}\
		&\rho&\\
		\text{s.t.}\
		&f(\bX)-\rho
		g(x)+\mH(p(\bX,\by))+\sum_{j=1}^s\eta_j\varphi_j(x)\in
		\bm_{\mathbf{d}}^{\sdsos}(Q),
		&\\
		&\rho\in\RR,\ \mH\in\left(\ckd{k}\right)^*,\ \eta\in\RR_+^s.&
	\end{aligned}\right.\tag{$\mathrm{P}^{\sdsos}_k$}
\end{equation}
It is obvious that $r_k^{\dsos}\le r_k^{\sdsos}\le r_k^{\psdp}\le
r^{\star}$ for each $k\in\N$, even in absence of the convexity
assumption in {\bf A1-2} (see Remark \ref{rk::lower}). It is
remarkable that the semidefinite
constraints brought by $\bm_{\mathbf{d}}(Q)$ in
\eqref{eq::f*rlower} are replaced by a set of linear inequality
constraints (resp. second-order cone constraints) in \eqref{eq::dsos}
(resp. \eqref{eq::sdsos}). Although the convergence of
$\{r_k^{\dsos}\}_{k\in\N}$ (resp. $\{r_k^{\sdsos}\}_{k\in\N}$) to
$r^{\star}$ is not guaranteed, the computation time for solving
\eqref{eq::dsos} (resp. \eqref{eq::sdsos}) could be considerably less
than that of \eqref{eq::f*rlower}. Consequently, for \eqref{FMP}
problems with large $m$ and $\mathbf{d}$ that are significantly
beyond the capability of the SDP relaxation \eqref{eq::f*rlower}, we
can still expect to obtain meaningful lower bounds of $r^{\star}$ in a
reasonable time by solving the alternatives \eqref{eq::dsos} or
\eqref{eq::sdsos} (see Example \ref{ex::compare2}).

\section{Convergence rate analysis}\label{sec::rate}

In this section, we consider the convergence rate of the lower bound
$r^{\dsdp}_k$ to the optimal value $r^{\star}$ and the outer
approximation $\Lambda_k$ to the feasible set $\K$. This will be done
by combining the convergence analysis of Lasserre's measure-based
upper bounds for polynomial minimization problems in \cite{SL2020} and
the metric regularity of
semi-infinite convex inequality system (c.f. \cite{CKLP2007}).
In this
section, to apply the results in \cite{SL2020}, {\itshape we assume
that the measure $\nu$ \eqref{eq::outerapp} is the Lebesgue
measure with support exactly $\Y$}.
%

Define the set-valued mapping $\mathcal{G}:\ \RR^m \
\rightrightarrows\ \RR^2$ by 
\[
		\mathcal{G}(x):=\{(\eta, R)\in\RR^2 \mid \Vert x\Vert\le
			R,\ p(x,y)\le \eta,\ \forall \ y\in\Y\}. 
\]
Let $\mathbf{F}:=\{x\in\RR^m \mid \varphi_i(x)\le 0, \ i=1,\ldots, s\}$. 
Then, it is clear that
$\K\cap\mathbf{B}^m_{\mathfrak{R}}=\mathcal{G}^{-1}(0,
\mathfrak{R})\cap\mathbf{F}$.

\begin{proposition}\cite[Lemma 3]{CKLP2007}\label{prop::metric}
The following statements are equivalent:
\begin{enumerate}[\upshape (i)]
	\item there exists $\bar{x}\in\RR^m$ such that $\Vert
		\bar{x}\Vert<\bar{R}$ and $p(\bar{x},y)<\bar{\eta}$ for all $y\in\Y$;
	\item $\mathcal{G}$ is metrically regular at any
		$u\in\mathcal{G}^{-1}(\bar{\eta}, \bar{R})$ for $(\bar{\eta},
		\bar{R})$, i.e., there exist $d_1,
		d_2>0$ and $c\ge 0$ such that whenever $\Vert x-u\Vert<d_1$
		and $\Vert(\eta, R)-(\bar{\eta}, \bar{R})\Vert<d_2$, it holds
		that
		\[
			\dist(x,\ \mathcal{G}^{-1}(\eta, R))\le c\cdot\dist((\eta,
			R),\ \mathcal{G}(x)). 
		\]
\end{enumerate}
\end{proposition}

Consider the assumption that 

\noindent {\bf A4:}\ There exists $\bar{x}\in\RR^m$ such that $\Vert
	\bar{x}\Vert<\mathfrak{R}$ and $p(\bar{x},y)<0$ for all $y\in\Y$.

\begin{corollary}\label{cor::dist}
	Under {\bf A4}, there exist $d>0$ and $c\ge 0$ such that
whenever $\Vert x\Vert\le \mathfrak{R}$ and $\dist(x,\
\mathcal{G}^{-1}(0, \mathfrak{R}))<d$, it holds that
\[
	\dist(x,\ \mathcal{G}^{-1}(0, \mathfrak{R}))\le
	c\cdot\max\left\{\max_{y\in\Y} p(x,y),\ 0\right\}.
\]
\end{corollary}
\begin{proof}
Note that for any $x\in\RR^m$ with $\Vert x\Vert\le \mathfrak{R}$,
$(\max_{y\in\Y} p(x,y), \mathfrak{R})\in\mathcal{G}(x)$. Then, 
for any $u\in\mathcal{G}^{-1}(0, \mathfrak{R})$, by {\bf A4} and Proposition
\ref{prop::metric}, there exist $d_u>0$ and $c_u\ge 0$ such that
whenever $\Vert x-u\Vert<d_u$ and $\Vert x\Vert\le \mathfrak{R}$, it
holds that 
\begin{equation}\label{eq::dist}
	\begin{aligned}
		\dist(x,\ \mathcal{G}^{-1}(0, \mathfrak{R}))&\le
		c_u\cdot\dist((0, \mathfrak{R}),\ \mathcal{G}(x))\\
		&\le c_u\cdot\max\left\{\max_{y\in\Y} p(x,y),\ 0\right\}. 
	\end{aligned}
\end{equation}
As $\mathcal{G}^{-1}(0, \mathfrak{R})$ is compact, we can find
finitely many points $u^{(i)}\in\mathcal{G}^{-1}(0, \mathfrak{R})$
and corresponding $d_{u^{(i)}}>0,\ c_{u^{(i)}}\ge 0$, $i=1,\ldots,t$, satisfying
\eqref{eq::dist} and 
\[
	\mathcal{G}^{-1}(0, \mathfrak{R})\subset\cup_{i=1}^t\{x\in\RR^m\mid \Vert
		x-u^{(i)}\Vert<d_{u^{(i)}}\}=:\mathcal{O}. 
\]
Moreover, there exists $d>0$ such that 
\[
	\{x\in\RR^m \mid \dist(x,\
		\mathcal{G}^{-1}(0, \mathfrak{R}))<d\}\subset\mathcal{O}. 
\]
Otherwise, there exists a sequence $\{x^{(k)}\}_{k\in\N}$ such that
$\dist(x^{(k)},\ \mathcal{G}^{-1}(0, \mathfrak{R}))<\frac{1}{k}$ and
$x^{(k)}\not\in\mathcal{O}$ for each $k\in\N$. As
$\mathcal{G}^{-1}(0, \mathfrak{R})$ is compact, we can assume that
there is a point $x'\in\RR^m$ such that $\lim_{k\rightarrow\infty}
x^{(k)}=x'$. Then, $\dist(x',\ \mathcal{G}^{-1}(0, \mathfrak{R}))=0$
and hence $x'\in\mathcal{G}^{-1}(0, \mathfrak{R})$. However, as
$\mathcal{O}$ is open, we have $x'\not\in\mathcal{O}$, a
contradiction. Then, the conclusion holds for this $d>0$ and
$c=\max_{1\le i\le t}c_{u^{(i)}}$. 
\end{proof}

For any $\mL_k\in(\RR[x]_{2\mathbf{d}})^*$, $k\in\N$, satisfying
the conditions in \eqref{eq::outerappro}, we define a number
$E(\mL_k):=\td{p}^{\star}_k-p^{\star}_k$, where 
\begin{equation}\label{eq::p*}
	\begin{aligned}
		p^{\star}_k:&=\min_{y\in\Y} -\mL_k(p(x,y))\\
		&=\min_{\mu\in\mathcal{M}(Y)} \int_{\Y}
		-\mL_k(p(x,y))\ud\mu(y)\quad \text{s.t. }\int_{\Y}\ud\mu(y)=1,
	\end{aligned}
\end{equation}
and 
\begin{equation}\label{eq::pp*}
	\left\{
	\begin{aligned}
		\td{p}^{\star}_k:= \min_{\sigma}&\ \int_{\Y}
		-\mL_k(p(x,y))\sigma(y)\ud y\\
		\text{s.t.}&\ \sigma(y)\in\Sigma^2_k[y],\ \int_{\Y}
		\sigma(y)\ud y=1. 
	\end{aligned}\right.
\end{equation}
In fact, \eqref{eq::pp*} is the $k$-th Lasserre's measure-based
relaxation (see \cite{LasserreOuter}) of \eqref{eq::p*}, where the
probability measures are replaced by the one having a density
$\sigma\in\Sigma^2_k[y]$ with respect to the Lebesgue measure.
Thus, $\td{p}^{\star}_k$ is an upper bound of $p^{\star}_k$ and
$E(\mL_k)\ge 0$. By the definition of
$\mathcal{P}^k(\Y)$, we have $\td{p}^{\star}_k\ge 0$. Hence, it holds
that
\[
	\max_{y\in\Y} \mL_k(p(x,y)) = -p_k^{\star} = E(\mL_k) -
	\td{p}^{\star}_k \le E(\mL_k).
\]
Clearly, $\mL_k^{\star}/\mL_k^{\star}(1)\in(\RR[x]_{2\mathbf{d}})^*$
satisfies the conditions in \eqref{eq::outerappro} for any minmizer
$\mL_k^{\star}$ of \eqref{eq::f*rduallower}.

\begin{theorem}\label{th::conrate}
	Under {\bf A1} and {\bf A4}, there exist $k'\in\N$ and $c\ge
	0$ such that whenever $k\ge k'$, 
\begin{equation}\label{eq::r1}
	\dist(\mL_k(x),\ \K\cap\mathbf{B}^m_{\mathfrak{R}})\le c\cdot E(\mL_k),
\end{equation}
for any $\mL_k$ satisfying the conditions in \eqref{eq::outerappro}.
Furthermore, under {\bf A1-4}, there exists $\td{c}\ge 0$ such
that whenever $k\ge k'$,
\[
	0\le r^{\star}-r^{\dsdp}_k\le \td{c}\cdot E(\mL_k^{\star}/\mL_k^{\star}(1)),
\]
where $\mL_k^{\star}$ is any minimizer of \eqref{eq::f*rduallower}. 
\end{theorem}
\begin{proof}
	Let $d,\ c$ be the numbers in Corollary \ref{cor::dist}.
	By Theorem \ref{th::outersdr}, the nested compact sets
	$\{\Lambda_{k}\}_{k\in\N}$ converges to
	$\K\cap\mathbf{B}^m_{\mathfrak{R}}$ as $k\to\infty$ in the
Hausdorff sense.  So there exists $k'\in\N$ such that
	whenever $k\ge k'$, $\dist(\mL_k(x),\
	\K\cap\mathbf{B}^m_{\mathfrak{R}})<d$ for any $\mL_k$ satisfying
	the conditions in \eqref{eq::outerappro}.

For any $\mL_k$ satisfying the conditions in \eqref{eq::outerappro}, as
$\varphi_i(x)$, $p(x,y)$ is sos-convex in $x$ for all
$i=1,\ldots,s$, $y\in\Y$,
Proposition \ref{prop::Jensen} implies that
\[
	\begin{aligned}
		&\varphi(\mL_k(x))\le \mL_k(\varphi) \le 0,\quad
		i=1,\ldots,s,\\
		&p(\mL_k(x),y)\le \mL_k(p(x,y)) \le
		E(\mL_k),\quad \text{ for all } y\in\Y.  
	\end{aligned}
\]
Hence, $\mL_k(x)\in\mathbf{F}$ and $\dist(\mL_k(x),\ \mathcal{G}^{-1}(0,
\mathfrak{R}))=\dist(\mL_k(x),\ \K\cap\mathbf{B}^m_{\mathfrak{R}})<d$ whenever $k\ge k'$. 
Recall that $\Lambda_k\subseteq\mathbf{B}^m_{\mathfrak{R}}$ for each
$k\in\N$ by Theorem \ref{th::outersdr}. Then according to Corollary
\ref{cor::dist}, \eqref{eq::r1} holds for any $k\ge k'$ because
\[
	\begin{aligned}
\dist(\mL_k(x),\ \K\cap\mathbf{B}^m_{\mathfrak{R}})&=\dist(\mL_k(x),\
\mathcal{G}^{-1}(0, \mathfrak{R}))\\
&\le c\cdot\max\left\{\max_{y\in\Y} p(\mL_k(x),y),\ 0\right\},\\
&\le c\cdot E(\mL_k).
\end{aligned}
\]

For each $k\in\N$, by Theorem \ref{th::lower}, $r^{\dsdp}_k$ is
attainable at a linear functional
$\mL^{\star}_k\in(\RR[x]_{2\mathbf{d}})^*$ feasible to
\eqref{eq::f*rduallower}. Then, by the sos-convexity of $f$ and
$-g$, 
\begin{equation}\label{eq::ineq}
	\frac{f(\mL^{\star}_k(x)/\mL^{\star}_k(1))}{g(\mL^{\star}_k(x)/\mL^{\star}_k(1))}\le
	\frac{\mL^{\star}_k(f)/\mL^{\star}_k(1)}{\mL^{\star}_k(g)/\mL^{\star}_k(1)}=r^{\dsdp}_k\le
	r^{\star}\le \frac{f(x)}{g(x)},\quad\text{for all }
	x\in\K\cap\mathbf{B}^m_{\mathfrak{R}}. 
\end{equation}
As $\mL_k^{\star}/\mL_k^{\star}(1)$ satisfies the conditions in
\eqref{eq::outerappro}, 
by the Lipschitz continuity of $\frac{f}{g}$ on
$\K\cap\mathbf{B}^m_{\mathfrak{R}}$, \eqref{eq::ineq} and
\eqref{eq::r1}, there exists $c'>0$ such that 
\[
	\begin{aligned}
		0\le r^{\star}-r^{\dsdp}_k&\le 
		\Big|\frac{f(\mL^{\star}_k(x)/\mL^{\star}_k(1))}{g(\mL^{\star}_k(x)/\mL^{\star}_k(1))}-\frac{f(x)}{g(x)}\Big|
		&\text{for all}\  x\in\K\cap\mathbf{B}^m_{\mathfrak{R}},\\
		&\le c'\cdot \dist(\mL^{\star}_k(x)/\mL^{\star}_k(1),\
		\K\cap\mathbf{B}^m_{\mathfrak{R}})\\
		&\le c'\cdot c\cdot E(\mL_k^{\star}/\mL_k^{\star}(1)). 
	\end{aligned}
\]
Letting $\td{c}=c'\cdot c$, the conclusion follows. 
\end{proof}

For each $k\in\N$, provided a  uniform bound of $E(\mL_k)$ for all
$\mL_k\in(\RR[x]_{2\mathbf{d}})^*$ satisfying the conditions in
\eqref{eq::outerappro}, which
is in term of $k$ but independent on $\mL_k$, we can establish the
convergence rate of $r^{\dsdp}_k$ and $\Lambda_k$ by Theorem
\ref{th::conrate}. We show that such bounds can be derived from the
paper \cite{SL2020} which investigates the
convergence analysis of Lasserre's measure-based upper bounds for
polynomial minimization problems. 

Since $\Y$ is compact, Proposition \ref{prop::bound} implies that
there are uniform bounds $B_1, B_2>0$ such that
\begin{equation}\label{eq::b12}
	\max_{y\in\Y} \Vert \nabla
	(-\mL_k(p(x,y)))\Vert\le B_1\quad\text{and}\quad
	\max_{y\in\Y}
	\Vert\nabla^2(-\mL_k(p(x,y)))\Vert\le B_2, 
\end{equation}
for all $\mL_k\in(\RR[x]_{2\mathbf{d}})^*, k\in\N$, satisfying the
conditions in \eqref{eq::outerappro}. 
Remark that the convergence analysis given in \cite{SL2020} depends
on the maximum norm of the gradient and Hessian of the
objective polynomial on the feasible set, rather than the
objective polynomial itself. Therefore,
the existence of $B_1$ and $B_2$
enables us to obtain the desired bounds of $E(\mL_k)$ by applying the
results in \cite{SL2020}. 
%
Next we only consider the case when $\Y$ is a general compact
subset of $[-1, 1]^n$ and satisfies 
\vskip 5pt
	\noindent {\bf A5:}\ \cite{KLZ2017} There exist constants $\epsilon_{\Y},
	\eta_{\Y}>0$ such that
	\[
		\text{\rm vol}(\mathbf{B}^n_{\delta}(y)\cap\Y)\ge\eta_{\Y}
		\text{\rm vol}(\mathbf{B}^n_{\delta}(y))=\delta^n\eta_{\Y}
		\text{\rm vol}(\mathbf{B}^n)\quad\text{for all}\ y\in\Y\
		\text{\rm and}\ 0<\delta<\epsilon_{\Y}. 
	\]
\noindent This is a rather mild assumption and satisfied by, for instance,
convex bodies, sets that are star-shaped with respect to a ball.
In this case, the following Proposition \ref{prop::el} can be drived 
straightforwardly from \cite[Theorem 10]{SL2020}. For completeness,
the proof is included in Appendix \ref{appendix}, which is almost a
repetition of the arguments in \cite{SL2020}. Denote
$\mathbf{H}^n:=[-1, 1]^n$.  

\begin{proposition}\label{prop::el}
	Under {\bf A5}, there exists a $k'\in\N$ such
	that whenever $k\ge k'$, 
	\[
		E(\mL_k)\le 2\sqrt{n}B_1\left(\frac{(4n+2)\log k}{\lfloor k/2\rfloor}
		+\frac{C}{k}\right)=O\left(\frac{\log k}{k}\right),\ \text{where}\quad 
		C=\frac{2^{3n+3}\vol(\mathbf{H}^n)}{\eta_{\Y}
		n^{n/2}\vol(\mathbf{B}^n)}, 
	\] 
	for all $\mL_k\in(\RR[x]_{2\mathbf{d}})^*$ satisfying the
	conditions in \eqref{eq::outerappro}. 
\end{proposition}

Theorem \ref{th::conrate} and Proposition \ref{prop::el} allow us to
state the following convergence rate of $\Lambda_k$ and $r_k^{\dsdp}$.
\begin{corollary}
	Under {\bf A1-5}, 
	as $k\rightarrow\infty$, 
\[
	\dist(u,\ \K\cap\mathbf{B}^m_{\mathfrak{R}})=O\left(\frac{\log
	k}{k}\right)\ \text{for all $u\in\Lambda_k$ and }
	0\le r^{\star}-r^{\dsdp}_k=O\left(\frac{\log k}{k}\right).
\]
\end{corollary}

\begin{remark}{\rm
		Moreover, thanks to the uniform bounds $B_1$ and $B_2$ in
		\eqref{eq::b12}, we can sharpen the above convergence
		rate in some special cases of $\Y$ using the results in
		\cite{SL2020}. For instance, the rate $O\left(\frac{\log
		k}{k}\right)$ can be improved to $O\left(\frac{\log^2
		k}{k^2}\right)$ when $\Y$ is a convex body and to
		$O\left(\frac{1}{k^2}\right)$ when $\Y$ is a simplex or
		ball-like convex body. For simplicity, the details
		are left to the interested readers.\qed
	}
\end{remark}

\section{Numerical experiments}\label{sec::num}
In this section, we present some numerical experiments to illustrate
the behavior of our SDP relaxation method for computing lower bounds
of $r^{\star}$ in \eqref{FMP}. 
All numerical experiments in the sequel were carried out on a PC with  
4-Core Intel i5 2GHz CPUs and 16G RAM. A rudimentary
Matlab code of our relaxation method and the experiment data can be
downloaded at \url{https://github.com/FengGuo2022/FSIPPsolve}.

In practice, to implement the SDP relaxations \eqref{eq::f*rlower} and
\eqref{eq::f*rduallower}, we need compute effectively the integrals
$\int_{\Y}y^{\beta}\ud\nu(y)$, $\beta\in\N^n$ to get the linear matrix
inequality representation of $\ckd{k}$ as mentioned in Section
\ref{sec::pre}. Here we list four cases of $\Y$ for which these
integrals can be obtained either explicitly in closed form or
numerically:
\begin{itemize}
	\item For $\Y=[-1, 1]^n$, we fix $\nu$ to be the Lebesgue measure on
		$\Y$. It is clear that
		\[
			\int_{\Y}y^{\beta}\ud\nu(y)=\left\{
				\begin{array}{ll}
					0& \text{if some } \beta_j\ \text{is odd},\\
\prod_{j=1}^n\frac{2}{\beta_j+1} & \text{if all } \beta_j\ \text{are
even.}
				\end{array}
				\right.
		\]
	\item For $\Y=\mathbb{S}_1:=\{y\in\RR^n\mid \Vert y\Vert=1\}$, we
		fix $\nu$ to be the $(n-1)$-dimensional surface measure. It was shown in
		\cite{Folland2001} that
		\[
				\int_{\Y}y^{\beta}\ud\nu(y)=\left\{
				\begin{array}{ll}
					0& \text{if some } \beta_j\ \text{is odd},\\
					\frac{2\Gamma(\hat{\beta}_1)\Gamma(\hat{\beta}_2)\cdots\Gamma(\hat{\beta}_n)}{\Gamma(\hat{\beta}_1+\hat{\beta}_2+\cdots+\hat{\beta}_n)} 
					& \text{if all } \beta_j\ \text{are even,}
				\end{array}
				\right.
		\]
		where $\Gamma(\cdot)$ is the gamma function and
		$\hat{\beta}_j=\frac{1}{2}(\beta_j+1)$, $j=1, \ldots , n.$
	\item For $\Y=\mathbf{B}^n=\{y\in\RR^n\mid \Vert y\Vert\le 1\}$,
		we fix $\nu$ to be the Lebesgue measure on $\Y$. It was shown
			in \cite{Folland2001} that
		\[
			\int_{\Y}y^{\beta}\ud\nu(y)=\frac{1}{\beta_1+\cdots+\beta_n+n}\int_{\mathbb{S}_1}y^{\beta}\ud\nu(y). 
		\]
	\item For a polytope $\Y\subset[-1, 1]^n$, we fix $\nu$ to be the
		Lebesgue measure on $\Y$. To get the integrals
		$\int_{\Y}y^{\beta}\ud\nu(y)$, we can use the software
		\texttt{LattE integrale} \cite{LattE} which is capble
		of {\itshape exactly} computing integrals of polynomials over convex polytopes. 
\end{itemize}

\begin{example}\label{ex::1}{\rm
		Now we provide four simple FSIPP problems
		\eqref{eq::ex1}-\eqref{eq::ex4} corresponding to
the above cases. It is easy to see that {\bf (A1-3)} hold for each
problem.  We use
the software {\sf Yalmip} \cite{YALMIP} to implement the SDP
relaxation \eqref{eq::f*rduallower} and call the SDP solver 
{\sf MOSEK} \cite{mosek}
to solve the resulting SDP problems. 
The standard semidefinite representation \eqref{eq::sdr} of
$\Lambda_k$ can be easily generated using {\sf Yalmip}. We draw
$\Lambda_k$ using the software package {\sf Bermeja} \cite{Bermeja}.
The computational results of the the SDP relaxations
\eqref{eq::f*rduallower} for each problem are shown in Table
\ref{tab::1} and \ref{tab::2}, including the approximate minimizers
$\mL^{\star}_k(x)/\mL^{\star}_k(1)$, lower bounds $r^{\dsdp}_k$ of
$r^{\star}$, as well as the CPU time, for $k=6,\ldots,15$. 
The SDr approximations $\Lambda_6$ of $\K$ in the
problem \eqref{eq::ex1}-\eqref{eq::ex4} are shown in Figure
\ref{fig}.

\vskip 8pt 
\noindent I. Consider the problem
\begin{equation}\label{eq::ex1}
	\left\{ \begin{aligned}
		\min_{x\in\RR^2}\
		&(x_1+1)^2+(x_2+1)^2\\ 
		\text{s.t.}\
		& p(x,y)=x_1^2+y_1^2x_2^2+2y_1y_2x_1x_2+x_1+x_2\le 0 ,\\ &
		\forall\ y\in [-1, 1]^2.  
	\end{aligned} \right.
\end{equation}
For any $y\in[-1,1]^2$, since $p(x,y)$ is of
degree $2$ and convex in $x$, it is sos-convex in $x$.
For any $x\in\RR^2$ and $y\in[-1, 1]^2$, it is clear that 
\[
p(x,y)\le  x_1^2+x_2^2+2|x_1x_2|+x_1+x_2.  
\] 
Then we can see
that the feasible set $\K$ can be defined by only two
constraints 
\[ 
	p(x,1,1)=(x_1+x_2)(x_1+x_2+1)\le 0\ \text{and}\
	p(x,1,-1)=(x_1-x_2)^2+x_1+x_2\le 0.  
\] 
That is, $\K$ is the area in $\RR^2$ enclosed by the ellipse
$p(x, 1,-1)=0$ and the
two lines $p(x, 1,1)=0$. Then, it is easy to check that the
only global minimizer of \eqref{eq::ex1} is
$u^{\star}=(-0.5,-0.5)$ and the minimum is $0.5$. 

\vskip 8pt

\noindent II. Consider the problem
\begin{equation}\label{eq::ex2}
	\left\{ \begin{aligned}
		\min_{x\in\RR^2}\
		&(x_1+1)^2+(x_2+1)^2\\ 
		\text{s.t.}\
		& p(x,y)=x_1^2+2y_1x_1x_2+(1-y_2^2)x_2^2+x_1+x_2\le 0,\\
&	\forall y\in\Y=\{y\in\RR^2\mid y_1^2+y_2^2\le 1\}.  
	\end{aligned} \right.
\end{equation}
For any $y\in\Y$, since $p(x,y)$ is of
degree $2$ and convex in $x$, it is sos-convex in $x$.
For any $x\in\RR^2$ and $\in\Y$, it is clear that 
\[
p(x,y)\le  x_1^2+x_2^2+2|x_1x_2|+x_1+x_2.  
\] 
Then we can see
that the feasible set $\K$ can be defined by only two
constraints 
\[ 
	p(x,1,0)=(x_1+x_2)(x_1+x_2+1)\le 0\ \text{and}\
	p(x,-1,0)=(x_1-x_2)^2+x_1+x_2\le 0.  
\] 
Thus, $\K$ is the same area as in Problem \eqref{eq::ex1}.
Hence, the only global minimizer of \eqref{eq::ex2} is
$u^{\star}=(-0.5,-0.5)$ and the minimum is $0.5$. 



\vskip 8pt
\noindent III.  Consider the problem
\begin{equation}\label{eq::ex3}
	\left\{ \begin{aligned}
		\min_{x\in\RR^2}\
		&(x_1-1)^2+(x_2-1)^2\\ 
		\text{s.t.}\
		& p(x,y)=\frac{(y_1x_1-y_2x_2)^2}{4}+(y_2x_1+y_1x_2)^2-1\le 0,\\
	&\forall y\in\Y=\{y\in\RR^2\mid y_1^2+y_2^2=1\}.
	\end{aligned} \right.
\end{equation}
Geometrically, the feasible region $\K$ is the common area of these
shapes in the process of rotating the ellipse defined by
$x_1^2/4+x_2^2\le 1$ continuously around the origin by $90^{\circ}$ clockwise.
Hence, $\K$ is the closed unit disk in $\RR^2$.
Then, it is not hard to check that the
only global minimizer of \eqref{eq::ex3} is
$u^{\star}=\left(\frac{\sqrt{2}}{2},\frac{\sqrt{2}}{2}\right)$ and the
minimum is $2\left(\frac{\sqrt{2}}{2}-1\right)^2\approx 0.1716$. 

\vskip 8pt
\noindent IV. Consider the problem
\begin{equation}\label{eq::ex4}
	\left\{ \begin{aligned}
		\min_{x\in\RR^2}\
		&(x_1+1)^2+(x_2-1)^2\\ 
		\text{s.t.}\
		&p(x,y)=-1+2x_1^2+2x_2^2-(y_1-y_2)^2x_1x_2\le 0,\\
		&\forall y\in\Y=\{y\in\RR^2\mid y_1\ge -1,\ y_2\le 1,\ y_2-y_1\ge 0\}.
	\end{aligned} \right.
\end{equation}
It is easy to see that $\K$ is in fact the area
enclosed by the lines $\sqrt{2}(x_2-x_1)=\pm 1$ and the circle 
$\{x\in\RR^2\mid 2x_1^2+2x_2^2=1\}$.
Hence, it is not hard to check that the
only global minimizer of \eqref{eq::ex4} is
$u^{\star}=\left(-\frac{\sqrt{2}}{4},\frac{\sqrt{2}}{4}\right)\approx
(-0.3536, 0.3536)$ and the
minimum is $2\left(\frac{\sqrt{2}}{4}-1\right)^2\approx 0.8358$. 
\qed

\begin{figure}
	\centering
	\subfigure[Problem \eqref{eq::ex1}]{
\scalebox{0.4}{
	\includegraphics[trim=80 200 80 200,clip]{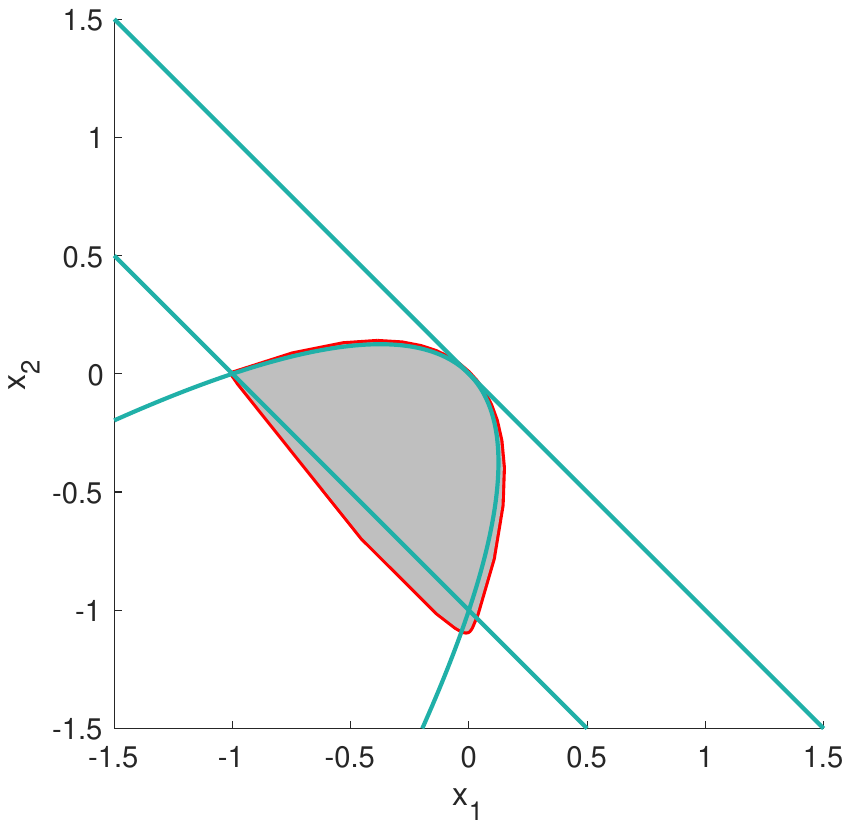}
}
}
\subfigure[Problem \eqref{eq::ex2}]{
\scalebox{0.4}{
	\includegraphics[trim=80 200 80 200,clip]{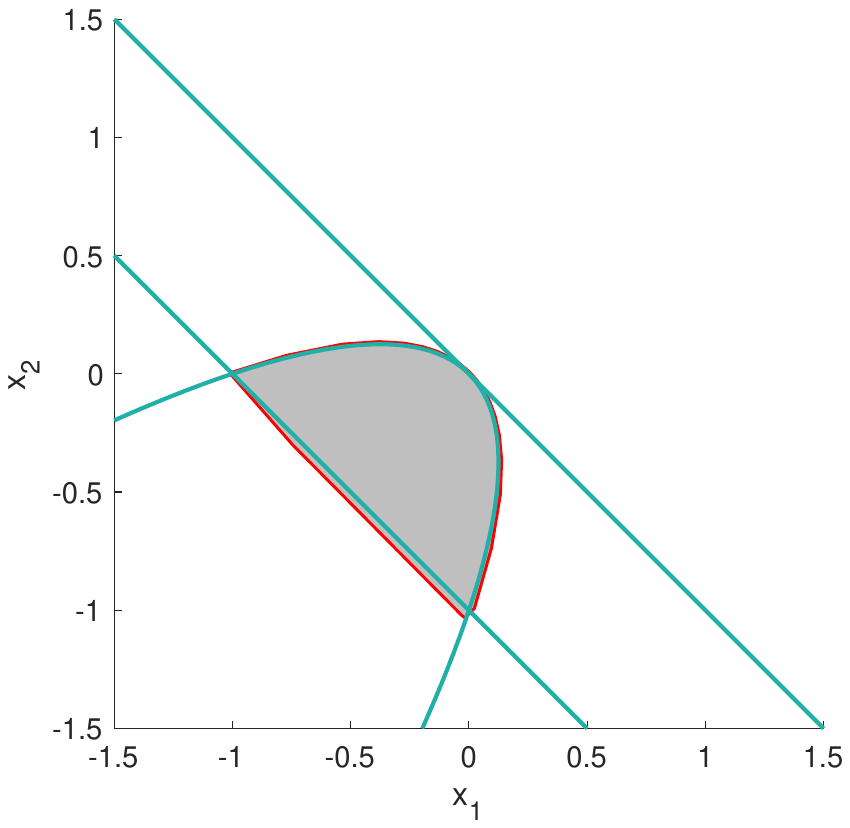}
}}\\
\subfigure[Problem \eqref{eq::ex3}]{
\scalebox{0.4}{
	\includegraphics[trim=80 200 80 200,clip]{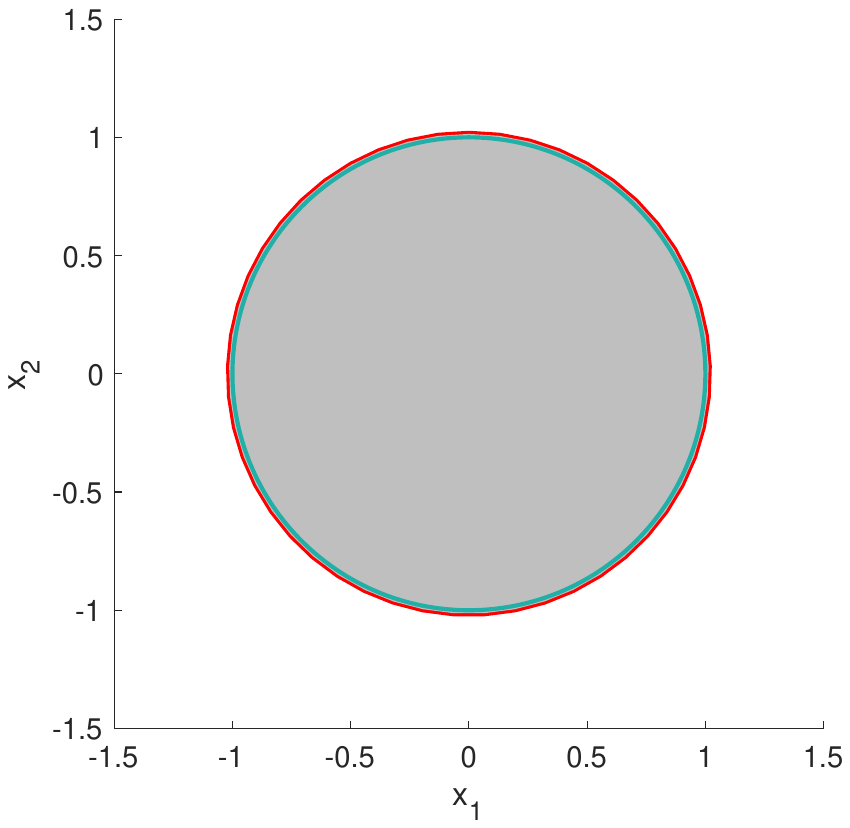}
}}
\subfigure[Problem \eqref{eq::ex4}]{
\scalebox{0.4}{
	\includegraphics[trim=80 200 80 200,clip]{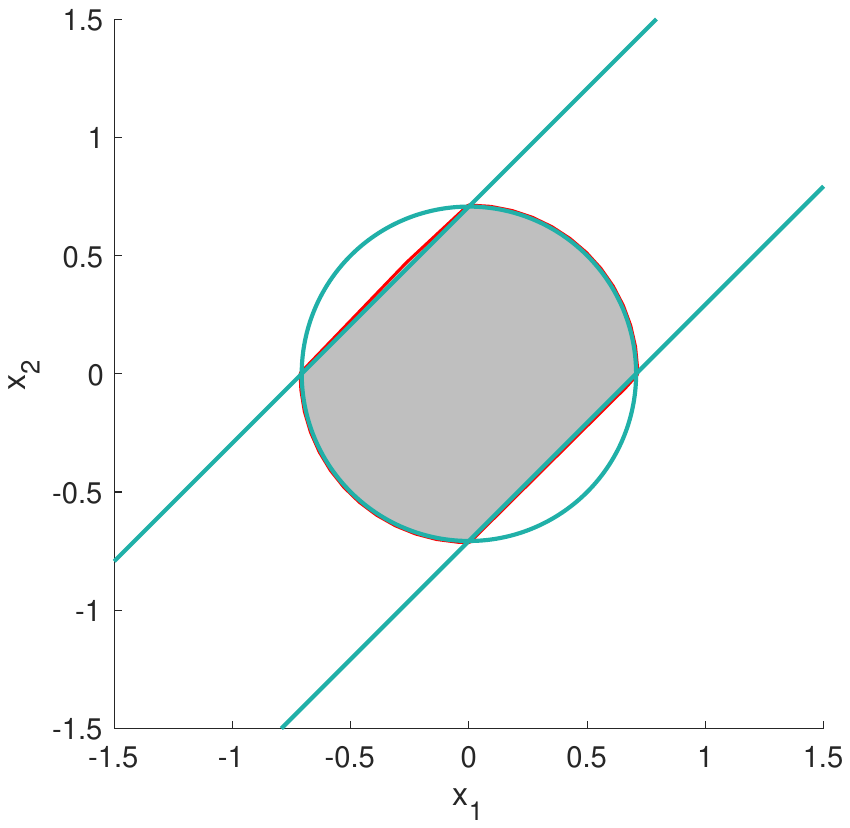}
}}
\caption{The feasible sets $\K$ (enclosed by the green curves) and
	their SDr approximations $\Lambda_6$ (the gray areas enclosed by the red
	curves) in Problem \eqref{eq::ex1}-\eqref{eq::ex4}.}
\label{fig}
\end{figure}

\begin{table}[htb]{\small
	\centering
	\begin{tabular}{rcccrcccr}
\midrule[0.8pt]
\multirow{2}{*}{$k$} &&\multicolumn{3}{c}{Problem \eqref{eq::ex1}}&
&\multicolumn{3}{c}{Problem \eqref{eq::ex2}} \\
		\cline{3-5} \cline{7-9}
&&$\mL^{\star}_k(x)/\mL^{\star}_k(1)$ & $r^{\dsdp}_k$& time& &
$\mL^{\star}_k(x)/\mL^{\star}_k(1)$ & $r^{\dsdp}_k$& time\\
\midrule[0.4pt]
		$6$&&$(-0.5368, -0.5964)$& $0.3775$&0.9s&&$(-0.5158, -0.5364)$&
		$0.4494$&1.0s\\
		$7$&&$(-0.5280, -0.5780)$& $0.4009$&1.3s&&$(-0.5121, -0.5289)$&
		$0.4600$&1.4s\\
		$8$&&$(-0.5220, -0.5644)$& $0.4182$&1.9s&&$(-0.5096, -0.5235)$&
		$0.4676$&2.0s\\
    9&&   $(-0.5178, -0.5541)$&    0.4314&   3.0s& &  $(-0.5078,
	-0.5195)$&    0.4732 &  3.1s\\
   10&&   $(-0.5147, -0.5461)$&    0.4416&   4.3s &&  $(-0.5065,
   -0.5164)$&    0.4775 &  4.5s\\
   11&&   $(-0.5123, -0.5397)$&    0.4497&   6.6s &&  $(-0.5054,
   -0.5140)$&    0.4808 &  7.1s\\
   12&&   $(-0.5105, -0.5346)$&    0.4562&   11.2s &&  $(-0.5046,
   -0.5121)$&    0.4834 &  10.7s\\
   13&&   $(-0.5092, -0.5306)$&    0.4612&   19.6s &&  $(-0.5041,
   -0.5106)$&    0.4854 &  18.1s\\
   14&&   $(-0.5090, -0.5297)$&    0.4623&   27.0s &&  $(-0.5037,
   -0.5095)$&    0.4869 & 29.2s\\
   15&&   $(-0.5082, -0.5277)$&    0.4649&  42.4s &&  $(-0.5036,
   -0.5089)$&    0.4877 & 43.8s\\
\midrule[0.8pt]
	\end{tabular}}
	\caption{Computational results for Problem
		\eqref{eq::ex1}-\eqref{eq::ex2}.}
		\label{tab::1}
\end{table}

\begin{table}[htb]{\small
	\centering
	\begin{tabular}{rcccrcccr}
\midrule[0.8pt]
\multirow{2}{*}{$k$} &&\multicolumn{3}{c}{Problem \eqref{eq::ex3}}&
&\multicolumn{3}{c}{Problem \eqref{eq::ex4}} \\
		\cline{3-5} \cline{7-9}
&&$\mL^{\star}_k(x)/\mL^{\star}_k(1)$ & $r^{\dsdp}_k$& time& &
$\mL^{\star}_k(x)/\mL^{\star}_k(1)$ & $r^{\dsdp}_k$& time\\
\midrule[0.4pt]
		$6$&&$(0.7174, 0.7174)$& $0.1597$&0.9s&&$(-0.3633, 0.3633)$&
		$0.8108$&67.3s\\
		$7$&&$(0.7152, 0.7152)$& $0.1622$&1.3s&&$(-0.3618, 0.3618)$&
		$0.8148$&114s\\
		$8$&&$(0.7137, 0.7137)$& $0.1640$&2.0s&&$(-0.3606, 0.3606)$&
		$0.8176$&171s\\
9&&   	$(0.7125, 0.7125)$&	0.1653&	2.9s&&	$(-0.3600, 0.3600)$&
0.8193&	257s\\
10&&	$(0.7117, 0.7117)$&	0.1663&	4.7s&&	$(-0.3596, 0.3596)$&
0.8203&	328s\\
11&&	$(0.7110, 0.7110)$&	0.1671&	7.1s&&	$(-0.3589, 0.3589)$&
0.8220&	465s\\
12&&	$(0.7105, 0.7105)$&	0.1677&	10.3s&&	$(-0.3584, 0.3584)$&
0.8232&	640s\\
13&&	$(0.7100, 0.7100)$&	0.1682&	17.3s&&	$(-0.3582, 0.3582)$&
0.8238&	862s\\
14&&	$(0.7097, 0.7097)$&	0.1686&	28.8s&&	$(-0.3579, 0.3579)$&
0.8246&	1147s\\
15&&	$(0.7094, 0.7094)$&	0.1689&	41.6s&&	$(-0.3576, 0.3576)$&
0.8255&	1508s\\
\midrule[0.8pt]
	\end{tabular}}
	\vskip 15pt
	\caption{Computational results for Problem
		\eqref{eq::ex3}-\eqref{eq::ex4}.}
		\label{tab::2}
\end{table}
}
\end{example}

\begin{remark}{\rm
	From Table \ref{tab::1}	and \ref{tab::2}, we can see that our new
	SDP method for \eqref{FMP} behaves similarly to Lasserre's
	measure-based SDP method for polynomial minimization problem
	\cite{LasserreOuter}. That is, the sequence of lower bounds
	$\{r^{\dsdp}_k\}_{k\in\N}$ increases rapidly in the first
	orders $k$, but rather slowly when close to $r^{\star}$. This
	behavior can also be expected from the convergence analysis
	discussed in Section \ref{sec::rate}. Nevertheless, the lower
	bounds obtained in a few orders indeed complement the upper bounds
	obtained by our previous work \cite{GuoJiaoFSIPP}. It is
	interesting to apply some acceleration techniques in
	\cite{LasserreOuter, LASSERRE2017, Lasserre2019} to
	improve the convergence rate of our new SDP method for
	\eqref{FMP}. We leave it for our future investigation.
}\qed
\end{remark}

The following example shows some computational behaviors of our SDP
method compared with the discretization method by grids (see
\cite{Hettich1993}) in computing lower bounds of $r^{\star}$. 
%

\begin{example}\label{ex::compare}
		{\rm Consider the problem
\begin{equation}\label{eq::exc}
	\left\{ \begin{aligned}
		r^{\star}:=\min_{x\in\RR^n}\ &\frac{\sum_{i=1}^n(x_i-1)^4}{\sum_{i=1}^n
	x_i+1}\\ 
		\text{s.t.}\
		&p(x,y)=\sum_{i=1}^n\left(1-\frac{(y_i-a_i)^2}{4}\right)x_i^2-1\le 0,\\
		&\forall y\in\Y=[-1,1]^n,\quad \varphi(x)=-\sum_{i=1}^n x_i\le 0,
	\end{aligned} \right.
\end{equation}
where each $a_i$ is a random number drawn from the standard uniform
distribution on the interval $[-1,1]$. Obviously, the feasible set
$\K$ is the intersection of the unit ball in $\RR^n$ with the halfspace
defined by $\sum_{i=1}^n x_i\ge 0$. Hence, the unique minimizer of
\eqref{eq::exc} is $\left(\sqrt{\frac{1}{n}},\ldots,
\sqrt{\frac{1}{n}}\right)$ and the optimal value is
$n\left(\sqrt{\frac{1}{n}}-1\right)^4/(1+\sqrt{n})$. It is clear that
{\bf A1-3} hold for \eqref{eq::exc}. 

Next, for each $n\in\N$, we generate random $a_i$'s in \eqref{eq::exc}
and solve it by the SDP relaxation \eqref{eq::f*rlower} and the
discretization method \eqref{DFMP} with the regular grid \eqref{eq::grid} whose
optimal value is denoted by $r_N^{\mbox{\upshape\tiny dis}}$. 
For the SDP relaxation \eqref{eq::f*rlower}, we use
the software {\sf Yalmip} to implement it and call the SDP solver {\sf
MOSEK} to solve the resulting SDP problems. For the
finitely constrained problem \eqref{DFMP}, as
discussed at the end of Section \ref{subsec::SDP}, we have tried to solve
it by (a) the bisection method for quasiconvex optimizaiton
using the software {\sf CVXPY} \cite{cvxpy}; (b) the interior-point
algorithm for nonlinear programming implemented in the Matlab command
{\sf fmincon}; (c) the SDP reformulation \eqref{eq::SDP4DFMP} which is
implemented by {\sf Yalmip} and solved by {\sf MOSEK}. Our numerical
experiments showed that the strategy (b) is more efficient and stable
than the other two when $n$ is large, so we only report here the
numerical results obtained by applying the Matlab command {\sf fmincon} to
\eqref{DFMP}. The initial feasible point for {\sf fmincon} is set to
be $0$.

We would
like to test $n$ in \eqref{eq::exc} as large as possible for which we can gain
meaningful lower bounds of $r^{\star}$ with these two methods. 
Therefore, we only compute and compare the first lower bound
obtained by these methods, i.e., we let $k=1$ and $N=1$ in
\eqref{eq::f*rlower}
and \eqref{eq::grid}, respectively. The computational results are
shown in Table \ref{tab::exc}.
As we can see, as $n$ increases, 
our SDP relaxations \eqref{eq::f*rlower} need much less time than
the discretization method in obtaining alike lower bounds.  
\qed
\begin{table}[h]
	\centering
	\begin{tabular}{cc r r r  c c}
\midrule[0.8pt]
$n$&&$r^{\psdp}_1$/time&& $r_1^{\mbox{\upshape\tiny dis}}$/time&& $r^{\star}$\\
\midrule[0.4pt]
10 &&  0.4414/3.7s  && 0.4752/0.9s		&& 0.5252\\[1pt]
11 &&  0.5209/4.4s  && 0.5254/2.6s		&& 0.6066\\[1pt]
12 &&  0.5959/5.4s  && 0.6042/8.1s		&& 0.6882\\[1pt]
13 &&  0.6360/6.8s  && 0.7152/22s		&& 0.7698\\[1pt]
14 &&  0.7438/9.5s  && 0.7565/1m13s		&& 0.8511\\[1pt]
15 &&  0.8109/18s   && 0.8224/3m50s		&& 0.9321\\[1pt]
16 &&  0.9050/23s   && 0.8815/12m40s	&& 1.0125\\[1pt]
17 &&  0.9343/29s   && 0.9824/37m34s	&& 1.0924\\[1pt]
18 &&  1.0070/45s	&& 1.0835/2h13m		&& 1.1716\\[1pt]
\midrule[0.8pt]\\
\end{tabular}
\caption{Computational results for Example
	\ref{ex::compare}.}\label{tab::exc}
\end{table}
}
\end{example}

\begin{remark}{\rm
	Remark that for the regular grids \eqref{eq::grid} in the
	discretization scheme, there are $(N+1)^n$ constaints to be generated
	from the grid points.
	The process could be very costly and the resulting problems become
	intractable for a large $n$.
	Of course, we are aware that there are other (commercial) softwares,
	which can deal with the process more efficiently
	and solve the resulting finitely constrained problems of much larger
	size. Meanwhile, the size of the semidefinite matrix in
	\eqref{eq::f*rlower} grows as
	$3\binom{m+\mathbf{d}}{m}+\binom{n+k}{n}$ and also becomes rapidly
	prohibitive as the order $k$ increases. Therefore, in view of the
	present status of available semidefinite solvers, we do not simply
	claim by Example \ref{ex::compare} any computational superiority of our SDP
	relaxation method over the discretization scheme. Instead, we intend
	to illustrate by the encouraging results in Example
	\ref{ex::compare} that our SDP relaxation method is promising to
	compute meaningful lower bounds of $r^{\star}$ with higher dimensional
	$\Y$ in a reasonable time. 
}\qed
\end{remark}

We end this paper with the following example to highlight the
scalability of the alternatives \eqref{eq::dsos} and
\eqref{eq::sdsos}.
\begin{example}\label{ex::compare2}
		{\rm Consider the problem
\begin{equation}\label{eq::exc2}
	\left\{ \begin{aligned}
		r^{\star}:=\min_{x\in\RR^m}\ &\sum_{i=1}^m(x_i-3)^2\\ 
		\text{s.t.}\
		&p(x,y)=\sum_{i=1}^mx_i^d-(1-y_1y_2)\le 0,\\
		&\forall y\in\Y=\{y\in\RR^2\mid y_1^2+y_2^2=1\},
	\end{aligned} \right.
\end{equation}
where $d\in\N$ is even. 
Clearly, the feasible set $\K$ is $\{x\in\RR^m\mid
	\sum_{i=1}^mx_i^d\le 1/2\}$. Hence, the unique minimizer of
	\eqref{eq::exc2} is $\left(\sqrt[d]{\frac{1}{2m}},\ldots,
	\sqrt[d]{\frac{1}{2m}}\right)$ and the optimal value is
	$m\left(\sqrt[d]{\frac{1}{2m}}-3\right)^2$. It is clear that
{\bf A1-3} hold for \eqref{eq::exc2}. 

Now we solve \eqref{eq::exc2} using the relaxations \eqref{eq::dsos},
\eqref{eq::sdsos} and \eqref{eq::f*rlower}. For comparison, we
implement all of the relaxations by means of the software package {\sf
spotless\_isos} \footnote{The {\sf spotless\_isos} software package is
available at:
\url{https://github.com/anirudhamajumdar/spotless/tree/spotless\_isos}.} 
\cite{AM2019} which is written using the Systems Polynomial
Optimization Toolbox \cite{spot}, and solve the resulting problems by
{\sf MOSEK}. As we are interested in comparing the impacts of
the dsos/sdsos/sos structures on the computation time and the obtained
lower bound quality of the corresponding relaxations, we fix the order
$k=1$ and let the numbers $(m, d)$ vary. The numerical results are
reported in Table \ref{tab::exc2}. 
Although the lower bounds $r^{\dsos}_1$ and
$r^{\sdsos}_1$ are not as good as $r^{\psdp}_1$, the times for
computing them are significantly less than that of $r^{\psdp}_1$. 
When $m$, $d$ are large and $r^{\psdp}_1$ is not available in a
reasonable time, we can still get meaningful lower bounds
$r^{\dsos}_1$ and $r^{\sdsos}_1$ by the alternatives \eqref{eq::dsos}
and \eqref{eq::sdsos}.

\begin{table}[h]
	\centering
	\begin{tabular}{rcrcrcrcr}
\midrule[0.8pt]
$(m, d)$&&$r^{\dsos}_1$/time&& $r^{\sdsos}_1$/time &&$r^{\psdp}_1$/time && $r^{\star}$\\
\midrule[0.4pt]
(16, 4) &&  83.50/2.5s		&& 102.79/3.3s	&&102.79/8.5s	&& 106.46\\[1pt]
(10, 6) &&  47.50/6.0s		&& 55.25/9.6s	&& 55.25/40s	&& 57.26\\[1pt]
(20, 4) &&  107.50/5.2s		&& 131.07/6.8s	&&131.07/59s	&& 135.44\\[1pt]
(12, 6) &&  59.50/15s		&& 67.41/24s	&&67.41/7m13s	&& 69.76\\[1pt]
(10, 8) &&  47.53/1m57s		&& 51.83/2m39s	&&51.83/8h28m		&& 53.47\\[1pt]
(12, 8) &&  59.57/7m9s		&& 63.09/10m45s	&&/>10h			&& 65.02\\[1pt]
\midrule[0.8pt]\\
\end{tabular}
\caption{Computational results for Example
	\ref{ex::compare2}.}\label{tab::exc2}
\end{table}

}
\end{example}

%

\section*{acknowledgements}
The authors are very grateful for the comments of two anonymous
referees which helped to improve the presentation.
The authors would like to thank M. J. C{\'a}novas and M. A. Goberna
for helpful comments on the metric regularity of semi-infinite convex
inequality system.
The authors are supported by the Chinese National Natural Science
Foundation under grant 11571350, the Fundamental Research
Funds for the Central Universities.

\appendix
\section{}\label{appendix}
We first recall some definitions and properties about the so-called needle
polynomials required in the proof of Proposition \ref{prop::el}.
\begin{definition}
	For $k\in\N$, the Chebyshev polynomial $T_k(t)\in\RR[t]_k$ is
	defined by  
	\[
		T_k(t)=\left\{
			\begin{aligned}
				&\cos(k \arccos t)	& \text{for}\ |t|\le 1,\\
				&\frac{1}{2}(t+\sqrt{t^2-1})^k+\frac{1}{2}(t-\sqrt{t^2-1})^k
				& \text{for}\ |t|\ge 1. 
			\end{aligned}
			\right.
	\]
\end{definition}

\begin{definition}\cite{KS1992}
For $k\in\N$, $h\in(0, 1)$, the needle polynomial
$v_k^h(t)\in\RR[t]_{4k}$ is defined by  
\[
	v_k^h(t)=\frac{T_k^2(1+h^2-t^2)}{T_k^2(1+h^2)}. 
\]
\end{definition}

\begin{theorem}\cite{KL2020, Kroo2015, KS1992}\label{th::np}
For $k\in\N$, $h\in(0, 1)$, the following properties hold for
$v_k^h(t)$:
\[
	\begin{aligned}
	& v_k^h(0)=1, & \\
	& 0\le v_k^h(t) \le	1 & \text{for}\ t\in [-1, 1], \\
	& v_k^h(t) \le 4e^{-\frac{1}{2}kh} & \text{for}\ t\in [-1, 1]\ \text{with}\ |t|\ge h. 
	\end{aligned}
\]
\end{theorem}
The following result gives a lower estimator which is used in the
proof of Proposition \ref{prop::el} to lower bound the integral of the
needle polynomial.
\begin{proposition}\cite[Lemma 13]{SL2020}\label{prop::le}
	Let $\phi(t)\in\RR[t]_k$ be a polynomial of degree up to $k\in\N$,
	which is nonnegative over $\RR_{\ge 0}$ and satisfies $\phi(0)=1,
	\phi(t)\le 1$ for all $t\in [0,1]$. Let $\Phi_k : \RR_{\ge 0}
	\rightarrow \RR_{\ge 0}$ be defined by 
	\[
		\Phi_k(t)=\left\{
		\begin{aligned}
			& 1-2k^2t & \text{if}\ t\le \frac{1}{2k^2},\\
			& 0 & \text{otherwise}.
		\end{aligned}
			\right.
	\]
	Then $\Phi_k(t) \le p(t)$ for all $t\in \RR_{\ge 0}$.
\end{proposition}

\noindent {\itshape Proof of Proposition \ref{prop::el}}\quad
For any $k\in\N$, let $\rho(k)=\frac{1}{16k^2}$ and $h(k):=(4n+2)\log
k/\lfloor k/2\rfloor$. Then, there exists a $k'\in\N$ such that
$\rho(k)\le
h(k)<\min\{\epsilon_{\Y}, 1\}$ for any $k\ge k'$. Fix a $k\ge k'$, a linear functional
$\mL_k\in(\RR[x]_{2\mathbf{d}})^*$ satisfying the conditions in
\eqref{eq::outerappro}, and a minimizer $y^{\star}$ of $\min_{y\in\Y} -\mL_k(p(x,y))$. 
Using the needle polynomial $v_k^h(t)\in\RR[t]$, define
$\sigma_k(y):=v_{\lfloor k/2\rfloor}^{h(k)}(\Vert
y-y^{\star}\Vert/(2\sqrt{n}))$.
Then, the polynomial $\td{\sigma}:=\sigma_k/\int_{\Y}\sigma_k\ud
y\in\Sigma_k^2[y]$ and feasible to \eqref{eq::pp*}. Hence, by Taylor's
theorem,
\begin{equation}\label{eq::e1}
	\begin{aligned}
		E(\mL_k)&\le \frac{1}{\int_{\Y}\sigma_k(y)\ud
	y}\int_{\Y}-\mL_k({p(x,y)})\sigma_k(y)\ud y - p^{\star}_k\\
	& = \frac{1}{\int_{\Y}\sigma_k(y)\ud
y}\int_{\Y}\sigma_k(y)(-\mL_k({p(x,y)} - p^{\star}_k) \ud y \\
& \le  \frac{B_1}{\int_{\Y}\sigma_k(y)\ud
y}\int_{\Y}\sigma_k(y)\Vert y-y^{\star}\Vert \ud y
\end{aligned}
\end{equation}
Define two sets
\[
	\Y_1:=\mathbf{B}_{2\sqrt{n}h(k)}^n(y^{\star})\cap\Y\quad\text{and}\quad
	\Y_2:=\mathbf{B}_{2\sqrt{n}\rho(k)}^n(y^{\star})\cap\Y\subseteq \Y_1. 
\]
Then, 
\begin{equation}\label{eq::e2}
	\int_{\Y}\sigma_k(y)\ud y \ge \int_{\Y_1}\sigma_k(y)\ud y 
	\ge \int_{\Y_2}\sigma_k(y)\ud y. 
\end{equation}
As $\Y\subseteq\mathbf{H}^n$, 
\begin{equation}\label{eq::e3}
	\begin{aligned}
		\int_{\Y}\sigma_k(y)\Vert y-y^{\star}\Vert \ud y& =
		\int_{\Y_1}\sigma_k(y)\Vert y-y^{\star}\Vert \ud y +
		\int_{\Y\setminus \Y_1}\sigma_k(y)\Vert y-y^{\star}\Vert \ud
		y\\
		&\le 2\sqrt{n}h(k)\int_{\Y_1}\sigma_k(y)\ud y + 
		2\sqrt{n} \int_{\Y\setminus \Y_1}\sigma_k(y)\ud y.
	\end{aligned}
\end{equation}
By Theorem \ref{th::np}, we have $\sigma_k(y)\le 4e^{-\frac{1}{2}h(k)\lfloor
k/2\rfloor}$ for any $y\in\Y\setminus\Y_1$ and hence
\[
\int_{\Y\setminus \Y_1}\sigma_k(y)\ud y \le 4e^{-\frac{1}{2}h(k)\lfloor
k/2\rfloor}\cdot\vol(\Y\setminus \Y_1)\le  4e^{-\frac{1}{2}h(k)\lfloor
k/2\rfloor}\cdot\vol(\mathbf{H}^n). 
\]
Moreover, by Proposition \ref{prop::le}, we have 
\[
	\sigma_k(y)\ge \Phi_{2k}(\Vert
y-y^{\star}\Vert/2\sqrt{n})=1-8k^2(\Vert
y-y^{\star}\Vert/2\sqrt{n})\ge \frac{1}{2}, 
\]
for all $y\in\Y_2$. Therefore, {\bf A5} implies
that
\begin{equation}\label{eq::e4}
	\int_{\Y_2}\sigma_k(y)\ud y \ge \frac{1}{2}\vol(\Y_2)\ge
	\frac{1}{2}\eta_{\Y} 2^n n^{n/2} \rho(k)^n\vol(\mathbf{B}^n)
	=\frac{\eta_{\Y} n^{n/2}\vol(\mathbf{B}^n)}{2^{3n+1}k^{2n}}.
\end{equation}
Combining \eqref{eq::e1}-\eqref{eq::e4}, we obtain
\begin{equation}\label{eq::e5}
	\begin{aligned}
		E(\mL_k)&\le 2\sqrt{n}B_1\left(h(k) +4e^{-\frac{1}{2}h(k)\lfloor k/2\rfloor}
		\frac{2^{3n+1}k^{2n}\vol(\mathbf{H}^n)}{\eta_{\Y}
		n^{n/2}\vol(\mathbf{B}^n)}\right)\\
		&=2\sqrt{n}B_1(h(k)+Ce^{-\frac{1}{2}h(k)\lfloor k/2\rfloor} k^{2n})
	\end{aligned}
\end{equation}
The conclusion follows by substituting $h(k)=(4n+2)\log k/\lfloor
k/2\rfloor$ in \eqref{eq::e5}.\qed

%


\end{document}